\documentclass[11pt]{amsart} 
\usepackage[lmargin=1in,rmargin=1in,tmargin=1in,bmargin=1in]{geometry}
\usepackage{tikz}
\usepackage{graphicx, float, epstopdf}
\usepackage{hyperref, color}
\usepackage{comment}
\usepackage{centernot}
\usepackage{fancyhdr}
\usepackage[utf8]{inputenc}
\usepackage{amsfonts,amssymb,amsmath,amsthm,mathrsfs}
\usepackage{graphics, setspace}
\usepackage{braket}
\usepackage{mathtools}
\usepackage{longtable}
\usepackage{pythonhighlight}
\usepackage{datetime}
\usepackage{algorithmic,algorithm}
\numberwithin{equation}{section}
\numberwithin{figure}{section}
\allowdisplaybreaks[4]                        
 
\newtheorem{lemma}{Lemma}[section]
\newtheorem{theorem}{Theorem}[section]

\theoremstyle{definition}
\newtheorem{definition}{Definition}[section]
\newtheorem{remark}{Remark}[section]
\newcommand{\Z}{\mathbb{Z}}

\newcommand{\R}{\mathbb{R}}
\newcommand{\C}{\mathbb{C}}
\newcommand{\N}{\mathbb{N}}

\newcommand{\e}{\operatorname{e}}

\newcommand{\real}{\operatorname{Re}}
\newcommand{\imag}{\operatorname{Im}}

\newcommand{\modu}{\operatorname{mod}}

\newcommand{\Pri}{\mathbb{P}}
\newcommand{\one}{\mathbf{1}}
\usepackage{tikz}
\newcommand{\Ip}{{\mathbf{1}_\Pri}}

\begin{document}

\title{Exponential sums twisted by general arithmetic functions}

\author[A.~Dong]{Anji Dong}
\address{
Anji Dong: Department of Mathematics,
University of Illinois Urbana-Champaign,
Altgeld Hall, 1409 W. Green Street,
Urbana, IL, 61801, USA}
\email{anjid2@illinois.edu}
\author[N.~Robles]{Nicolas Robles}
\address{
Nicolas Robles: RAND Corporation, Engineering and Applied Sciences, 1200 S Hayes St, Arlington, VA, 22202, USA}
\email{nrobles@rand.org}
\author[A.~Zaharescu]{Alexandru Zaharescu}
\address{
Alexandru Zaharescu: Department of Mathematics,
University of Illinois Urbana-Champaign,
Altgeld Hall, 1409 W. Green Street,
Urbana, IL, 61801, USA and Simion Stoilow Institute of Mathematics of the Romanian Academy, 
P. O. Box 1-764, RO-014700 Bucharest, Romania}
\email{zaharesc@illinois.edu}
\author[D.~Zeindler]{Dirk Zeindler}
\address{
Dirk Zeindler: Department of Mathematics and Statistics, Lancaster University, Fylde College, Bailrigg, Lancaster LA1 4YF, United Kingdom}
\email{d.zeindler@lancaster.ac.uk}

\begin{abstract}
    We examine exponential sums of the form $\sum_{n \le X} w(n) \e(\alpha n^k)$, for $k=1,2$, where $\alpha$ satisfies a generalized Diophantine approximation and where $w$ are different arithmetic functions that might be multiplicative, additive, or neither. A strategy is shown on how to bound these sums for a wide class of functions $w$ belonging within the same ecosystem. Using this new technology we are able to improve current results on minor arcs that have recently appeared in the literature of the Hardy-Littlewood circle method. Lastly, we show how a bound on $\sum_{n \le X} |\mu(n)| \e (\alpha n)$ can be used to study partitions asymptotics over squarefree parts and explain their connection to the zeros of the Riemann zeta-function.
\end{abstract}


\maketitle
\section{Introduction} \label{sec:introduction}
\subsection{Motivation and results} \label{sec:Introduction}
Suppose that $\alpha \in \R$ and $r \in \N$. To put our results in context, we first set the following definition
\begin{align} \label{eq:def_S_r}
    S_r(\alpha,X) := \sum_{\substack{p_1,\ldots,p_r\in\Pri\\p_1\cdots p_r\leq X}} \e(\alpha p_1\cdots p_r) \quad \textnormal{where} \quad \e(x) = \exp(2 \pi i x) ,
\end{align}
and where $\Pri$ denotes the set of primes.
One of our central results is as follows.

\begin{theorem} \label{thm:mainresult}
    Let $\alpha\in\R$, $a\in\Z$, $q\in\N$ and $\Upsilon>0$ such that 
    \begin{align} \label{eq:approx_alpha_for_main section 1}
    \left|\alpha-\frac{a}{q}\right|\leq \frac{\Upsilon}{q^2}, \quad (a,q)=1.  
    \end{align}
    Then for any $X\geq 2, q \leq X$ and $r \in \N$, we have
    \begin{align} \label{eq:centralequation}
    S_r(\alpha,X)
    &\ll_r 
    \bigg( \frac{X}{q^{\frac{1}{2r}}} \max\{1,\Upsilon^{\frac{1}{2r}}\}+ X^{\frac{2+2r}{3+2r}} + X^{\frac{2r-1}{2r}}q^{\frac{1}{2r}}\bigg) (\log X)^3.
    \end{align}
\end{theorem}

We shall prove more general versions of this theorem as well as several related results involving exponential sums twisted by different functions of primes in the forthcoming sections. These twisted exponential sums can be thought of as belonging within the same ecosystem and the techniques we present can be repurposed for other types of exponential sums.

Bounds on exponential sums have several useful applications in different areas of number theory. One such notable application of a classic inequality in additive number theory takes place when $\Upsilon=r=1$ in Theorem \ref{thm:mainresult} and it is as follows.

\begin{theorem}[Vinogradov estimate] \label{eq:originalVinogradov}
    Let $\alpha \in \R$, $a \in \Z$ and $q \in \N$ such that
    $|\alpha-\frac{a}{q}| \leq \frac{1}{q^2}$ with $(a,q)=1$.
   Then for any $X\geq 2$, one has
    \begin{align}
    S_1(\alpha,X) \ll \left(\frac{X}{\sqrt{q}}+X^{\frac{4}{5}}+\sqrt{X}\sqrt{q}\right)(\log X)^3.
\end{align}
   
\end{theorem}

This was established by Vinogradov \cite{Vi42} and the proof was later simplified considerably by Vaughan \cite{Va77}, see also \cite[$\mathsection$25]{Da74} and \cite[$\mathsection$23]{primesDimitris}. In its original form it appeared as $\sum_{n \leq x} \Lambda(n) \e(\alpha n)$ where $\Lambda$ is the von Mangoldt function, a sum whose generalization we shall also cover in this paper. Vinogradov's bound is an essential step in bounding the minor arcs arising from the sums of three primes as well as in partitions into \textit{primes} as shown by Vaughan in \cite{VaughanPrimes}. Specifically, using this bound, Vinogradov was able to show that every sufficiently large odd integer is the sum of three primes and also obtained an asymptotic formula for the number of such representations. 

As of recently, a rich amount of literature on partitions involving the Hardy-Littlewood circle method has been published. Each case required its own specialized and associated exponential sum. For instance in \cite{dunnrobles, gafniPowers}, the associated exponential sum needed to bound partitions into \textit{powers and certain restricted polynomials} was Weyl's bound \cite[$\mathsection3$]{HLM} on $\sum_{n \leq x} \e(\alpha n^k)$, the sum being taken over integers $n$, not primes $p$. Weyl's bound was also instrumental in \cite{amitaPartitions} to study partitions in \textit{arithmetic progressions}. In \cite{gafniPrimePowers}, an analogous result of Kawada-Wooley \cite{kawadaWooley} was employed by Gafni to bound a sum over \textit{prime powers}. In \cite{colorsDivisor}, an extension of a classic result of Motohashi \cite{motohashiExponential} had to be derived to bound sums of the form $\sum_{n \leq X} \sigma_{k_1,k_2}(n) \e(\alpha n)$, where $\sigma_{k_1,k_2}$ is the \textit{generalized divisor function}. Another instance is in \cite{taylorMoebius}, where Daniels bounded exponential sums for \textit{signed partitions} employing Davenport's inequality for the M\"{o}bius function \cite{davenportInequality}, see also \cite{MV77}.

For asymptotic partitions into \textit{semiprimes} \cite{semiprimes}, a generalization of Theorem \ref{eq:originalVinogradov}
 was established.

\begin{theorem}[Generalized Vinogradov estimate, {\cite[$\mathsection$5.1]{semiprimes}}]
\label{thm:general_vinogradov_old}
Let $\alpha \in \R$,  $a\in\Z$, $q\in\N$ and $\Upsilon \geq 1$ such that $|\alpha-\frac{a}{q}| \leq \frac{\Upsilon}{q^2}$ with $(a,q)=1$. For any $X\geq 2$, one has
    \begin{align} \label{eq:oldUpsilon}
    S_1(\alpha,X) \ll \Upsilon\left(\frac{X}{\sqrt{q}}+X^{\frac{4}{5}}+\sqrt{X}\sqrt{q}\right)(\log X)^3.
    \end{align}
\end{theorem}	

This was the stepping stone to obtaining a bound for $S_2(\alpha,X)$ by the use of the so-called `hyperbola method'. Indeed, employing Theorem \ref{thm:general_vinogradov_old}, the following bound on $S_2$ was proved.

\begin{theorem}[{\cite[Theorem~5.1]{semiprimes}}]
    \label{prop:doublevinogradov}
    Let $\alpha \in \R$. If $a \in \Z$ and $q \in \N$ are such that $|\alpha - \frac{a}{q}| \leq \frac{1}{q^2}$ with $(a,q)=1$. For any $X\geq 2$, one has
    \begin{align} \label{eq:oldS2}
    S_2(\alpha, X) \ll 
    \frac{X}{q^{\frac{1}{6}}} (\log X)^{\frac{7}{3}} + X^{\frac{16}{17}} (\log X)^{\frac{39}{17}} + X^{\frac{7}{8}}q^{\frac{1}{8}} (\log X)^{\frac{9}{4}}.
    \end{align}
\end{theorem}

In this article, we go a good deal further and improve Theorem \ref{thm:general_vinogradov_old} by considerably simplifying the original argument in \cite{semiprimes} (which was based on Davenport's \cite[$\mathsection$25]{Da74}) and obtaining a substantially stronger version, along with a much more elegant proof. The improvements of Theorem \ref{thm:mainresult} over Theorem \ref{thm:general_vinogradov_old} are threefold:
\begin{itemize}
    \item the range of validity of $\Upsilon$ is extended from $[1,\infty)$ to $(0,\infty)$,
    \item only the first term inside the brackets on the right-hand side of \eqref{eq:centralequation} is affected by $\Upsilon$, rather than having $\Upsilon$ as global prefactor in \eqref{eq:oldUpsilon},
    \item the exponent of $\Upsilon$ is now reduced from $1$ to $\frac{1}{2}$.
\end{itemize}
In turn, these improvements lead to a stronger version of Theorem \ref{prop:doublevinogradov}, which we now illustrate for $r=2$ and $r=3$.

    \begin{theorem} \label{thm:S2S3}
    Let $\alpha \in \R$,  $a\in\Z$, $q\in\N$ and $\Upsilon>0$ such that $|\alpha-\frac{a}{q}|\leq \frac{\Upsilon}{q^2}$ with $(a,q)=1$. For any $X\geq 2$, one has
    \begin{align} 
    S_2(\alpha,X) &\ll \frac{X}{q^{\frac{1}{4}}} \max\{1,\Upsilon^{\frac{1}{4}}\} (\log X)^\frac{5}{2} + X^{\frac{6}{7}} (\log X)^\frac{19}{7} + X^{\frac{3}{4}}q^{\frac{1}{4}} (\log X)^\frac{5}{2}, \label{eq:newS2} \\
    S_3(\alpha,X)
    &\ll 
    \frac{X}{q^{\frac{1}{6}}} \max\{1,\Upsilon^{\frac{1}{6}}\} (\log X)^{\frac{7}{3}}
    +
    X^{\frac{8}{9}} (\log X)^{\frac{23}{9}}
    +
    X^{\frac{5}{6}} q^{\frac{1}{6}} (\log X)^{\frac{7}{3}}. \label{eq:newS3}
    \end{align}
\end{theorem}

For example, let us take the case $r=2$. The second and third exponents of $X$ on the right-hand side of \eqref{eq:oldS2} and \eqref{eq:newS2} are improved by
\begin{align} \label{eq:improvements}
    \frac{16}{17} - \frac{6}{7} = \frac{10}{119} \quad \textnormal{and} \quad
    \frac{7}{8} - \frac{3}{4} = \frac{1}{8},
\end{align}
respectively. These improvements are due to the fact that $\Upsilon$ appears with a $\frac{1}{2}$ exponent. However, the presence of $\Upsilon$ only in the first term of the right-hand of \eqref{eq:centralequation} also simplifies matters considerably, for otherwise using Theorem \ref{thm:general_vinogradov_old} with the hyperbola method would have yielded 
\begin{align} \label{eq:s3finalprimorial_r3}
S_3(\alpha,X)
\ll &\,
\Upsilon X^{\frac{1}{2}}(\log X)^2 q^{\frac{1}{2}}
+
\Upsilon X q^{-\frac{1}{18}} (\log X)^{\frac{19}{9}}
+
\Upsilon X^{\frac{52}{53}}  (\log X)^{\frac{111}{53}}
+
\Upsilon X^{\frac{25}{26}} q^{\frac{1}{26}} (\log X)^{\frac{27}{13}}
\nonumber\\
&+
\Upsilon \left(\log\log X\right) \big(X q^{-\frac{1}{6}} (\log X)^{\frac{7}{3}}  +  X^{\frac{7}{8}}q^{\frac{1}{8}} (\log X)^{\frac{9}{4}} \big),
\end{align}
which is not only weaker than \eqref{eq:newS3}, but substantially more complicated and cumbersome.

The reason Theorem \ref{thm:S2S3} is presented as a theorem, rather than as corollary of Theorem \ref{thm:mainresult}, is because we have opted to present Theorem \ref{thm:mainresult} in a compact notation where the exponents of the logarithms are a bit weaker. However, with some additional work, we can refine these exponents to those in Theorem \ref{thm:S2S3} not only for $r \in \{2,3\}$ but for any integer $r$. Although different applications might require specific precision in the exponents, in general, and certainly for our purposes, the exponents associated to the logarithms are not as critical as the exponents associated to $X$ and $q$. Therefore, when warranted, we have chosen to provide simplified arguments for the sake of clarity.

The techniques employed to prove Theorem \ref{thm:mainresult} and Theorem \ref{thm:S2S3} lend themselves well to study adjacent exponential sums twisted by other arithmetic functions. If we let $\Lambda^{*r}$ denote the $r$-fold Dirichlet convolution of $\Lambda$ with itself, then one such instance is
\begin{align} \label{eq:s3finalprimorial_lambda}
    \widetilde{S}_r(\alpha,X) := \sum_{\substack{n\leq X}}\Lambda^{*r}(n)\e(\alpha n) =
    \sum_{\substack{n_1,\ldots,n_r\in\N\\n_1\cdots n_r\leq X}}\e(\alpha n_1\cdots n_r)\prod_{i=1}^r\Lambda(n_i).
\end{align}
Along the way, we shall present analogue bounds for \eqref{eq:s3finalprimorial_lambda}. However, we do not stop here. 
\begin{definition}
    Let $f: \N \to \C$ be an arithmetic function. For any real $\alpha$ and $X \ge 1$ we shall adopt the notation 
\begin{align}
    S_f(\alpha,X) := \sum_{n \le X}f(n) \e(\alpha n).
\end{align}
\end{definition}

The notation $S_r$ is reserved exclusively for \eqref{eq:def_S_r}. For instance in \cite{BRZ23}, convolutions of M\"{o}bius functions were considered. In this direction, we have the following result.

\begin{theorem}\label{thm:mumu}
     Let $\alpha = a/q+\beta$ for some $(a,q)=1$ and $|\beta|<\Upsilon/q^2$ for some $\Upsilon>0$. 
     For any $X\geq 2$, one has
    \begin{align} \label{eq:new_mu_mu}
    S_{\mu*\mu} (\alpha,X) 
    \ll_\varepsilon 
    \frac{X}{q^{\frac{1}{4}}}\max\{1,\Upsilon^{\frac{1}{4}}\}(\log X)^{\frac{5}{2}} +X^{\frac{6}{7}+\varepsilon}+ X^{\frac{3}{4}}q^{\frac{1}{4}}(\log X)^{\frac{5}{2}}, 
    \end{align}
    for arbitrarily small $\varepsilon > 0$.
\end{theorem}

Theorem \ref{thm:mumu} improves the bound for $S_{\mu*\mu}$ that appeared in \cite{BRZ23}. 
Theorem \ref{thm:mumu} will be further generalized for the $r$-fold convolution $\mu^{*r}$ in Section \ref{subsec:aside}. If we let $\mathbf{1}_\Pri$ denote the characteristic function of primes, then Theorem \ref{thm:S2S3} shows a bound for exponential sums twisted by $\mathbf{1}_\Pri * \mathbf{1}_\Pri$ and $\mathbf{1}_\Pri * \mathbf{1}_\Pri * \mathbf{1}_\Pri$, whereas in Theorem \ref{thm:mumu} we studied exponential sums twisted by $\mu * \mu$. Therefore, this begs the question of bounding an exponential sum twisted by 
\begin{align}
    \mu_{_\Pri}(n) := (\mu*\mathbf{1}_\Pri)(n) = \sum_{\substack{hp=n \\ p \in \Pri}} \mu(h).
\end{align} 

We illustrate the behavior of $\mu_{_\Pri}$ in Figure \ref{fig:moebiusprime} where the color hue represents magnitude.

\begin{figure}[H]
    \centering
    \includegraphics[scale=0.43]{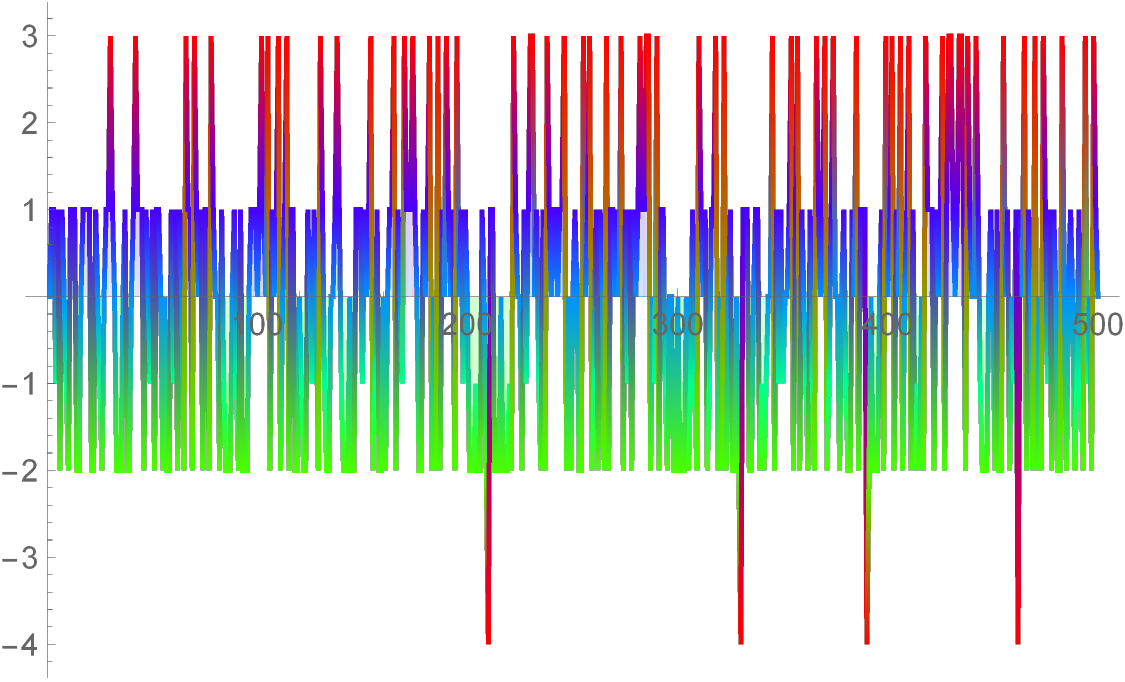}
    \includegraphics[scale=0.43]{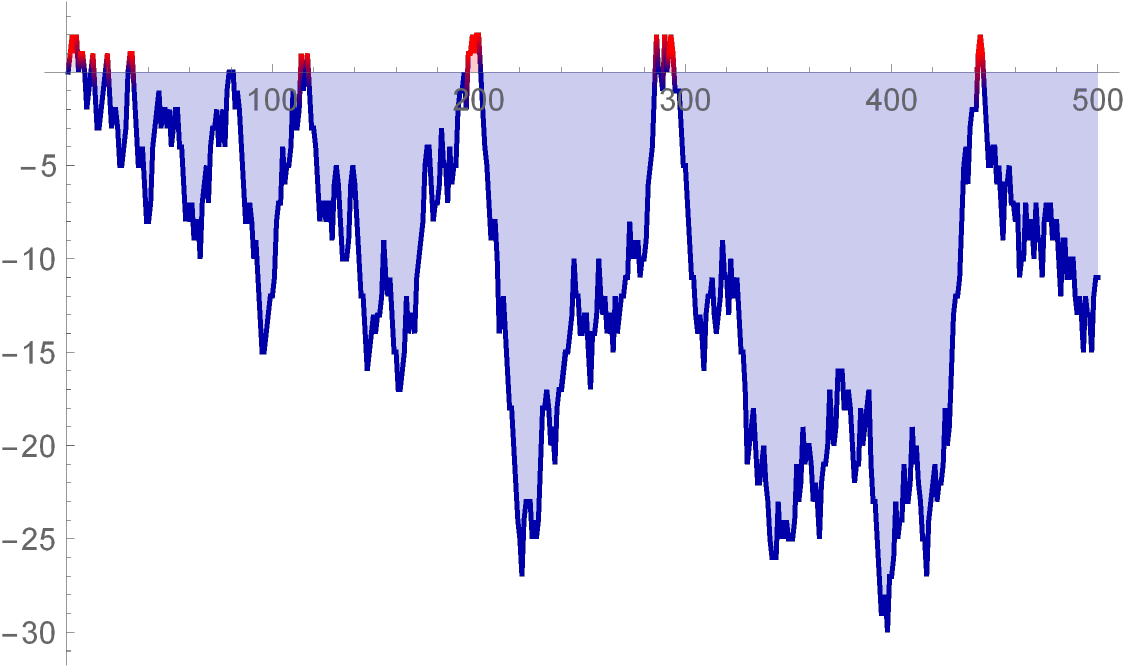}
    \caption{Values of $\mu_{_\Pri}(x)$ on the left, and $\sum_{n \le x}\mu_{_\Pri}(n)$ on the right, for $1 \le x \le 500$.}
    \label{fig:moebiusprime}
\end{figure}

The result is as follows.

\begin{theorem} \label{thm:muPri}
    Let $\alpha = a/q+\beta$ for some $(a,q)=1$ and $|\beta|<\Upsilon/q^2$ for some $\Upsilon>0$. 
    Then for any $X\geq 2$, one has
    \begin{align} \label{eq:new_mu_1p}
    S_{\mu_{_\Pri}} (\alpha,X) \ll_\varepsilon 
\frac{X}{q^{\frac{1}{4}}}\max\{1,\Upsilon^{\frac{1}{4}}\}(\log X)^{\frac{5}{2}} +X^{\frac{6}{7}+\varepsilon} + X^{\frac{3}{4}}q^{\frac{1}{4}}(\log X)^{\frac{5}{2}}, 
    \end{align}
for arbitrarily small $\varepsilon > 0$.
\end{theorem}

Lastly, a natural step is to consider the twist by $\mathbf{1} * \mathbf{1}_{\mathbb{P}}$, where $\mathbf{1}(n)=1$ for all $n$, which is incidentally $\omega(n)$, the number of distinct prime divisors of $n$. This bound will be revisited in \cite{gafniRoblesVaughan}.

\begin{theorem} \label{thm:Pri1}
   Let $\alpha = a/q+\beta$ for some $(a,q)=1$ and $|\beta|<\Upsilon/q^2$ for some $\Upsilon > 0$. 
   For any $X\geq 2$, one has
   \begin{align} \label{eq:new_1_1p}
      S_{\mathbf{1}*\mathbf{1}_{\mathbb{P}}} (\alpha,X) &\ll 
\frac{X}{q^{\frac{1}{4}}}\max\{1,\Upsilon^{\frac{1}{4}}\}(\log X)^{\frac{5}{2}} +X^{\frac{6}{7}}(\log X)^{\frac{19}{7}}+ X^{\frac{3}{4}}q^{\frac{1}{4}}(\log X)^{\frac{5}{2}}. 
   \end{align}
\end{theorem}

We now present a new bound for sums twisted by of $|\mu|$.

\begin{theorem}\label{thm:exponential sum for mu squared}
    Let $\alpha = a/q+\beta$ for some $(a,q)=1$ and $|\beta|<1/q^2$. 
    For $X\geq 2$, one has
    \begin{align} \label{eq:new_|mu|}
    S_{|\mu|} (\alpha,X) \ll\frac{X}{q}\log X+X^{\frac{8}{13}}(\log X)^{\frac{37}{13}}+q\log X.
    \end{align}
\end{theorem}

This sum had been studied in \cite{BrudenETAL, BrudenPerelli}. The latest results were published in \cite{puchta} where Schlage-Puchta showed that $S_{|\mu|} (\alpha,X) \ll (\frac{X}{q}+q)X^\varepsilon$ for $q \le Q \le \frac{1}{2}X^{\frac{1}{2}}$ and for all $\varepsilon>0$. Theorem \ref{thm:exponential sum for mu squared} extends the range of $Q$ from $Q \le \frac{1}{2}X^{\frac{1}{2}}$ to $Q \le X$, after which it is trivial, although it still holds. Moreover, the $\varepsilon$ term is removed and instead replaced by a power of $\log X$. As we shall see, this will allow us to bound the minor arcs arising from the partitions associated to $|\mu|$.

In all of our previous sums, the term $\e(\cdot)$ was of the form $\e(\alpha n)$. With the technique we present we can also study sums with $\e(\alpha n^2)$ instead of $\e(\alpha n)$ as the below result shows.

\begin{theorem} \label{thm:exponential sum for mu squared squared}
     Let $\alpha = a/q+\beta$ for some $(a,q)=1$ and $|\beta|<1/q^2$. 
     For $X\geq 2$, one has
    \begin{align*}
        \tilde S_{|\mu|}(\alpha,X) = \sum_{n \leq X} \mu^2(n) \e(\alpha n^2) \ll \frac{X}{q^{1/4}}+X^{1/2}\log X(\log q)^{1/2}+X^{1/2}q^{1/4}(\log q)^{1/4}.
    \end{align*}
\end{theorem}

We now end with the above mentioned application to partitions. The generating function for partitions weighted by $|\mu|$ is
\begin{align} \label{eq:generatingPartitions}
\Psi(z) = \sum_{n=0}^\infty \mathfrak{p}_{|\mu|}(n)z^n = \prod_{n=1}^\infty (1-z^n)^{-|\mu(n)|}.
\end{align}
The number $\mathfrak{p}_{|\mu|}(n)$ can be interpreted combinatorially as the number of partitions of $n$ with only square free parts.
An application of the Hardy-Littlewood circle method along with Theorem \ref{thm:exponential sum for mu squared} will allow us to prove the following result. 
\begin{theorem} 
\label{thm:partitionTheorem}
We have as $n \to \infty$ that
\begin{align}
\log\mathfrak{p}_{|\mu|}(n)
=
 2\sqrt{n} +  O_\varepsilon(n^{\frac{1}{4}+\varepsilon}),
\end{align}
for every $\varepsilon>0$. Further, if there exists a real $h$ with $\frac{1}{2}\leq h \leq 1$ such that the real part of all non-trivial zeros of $\zeta$ are less or equal to $h$ then 
\begin{align}
\log\mathfrak{p}_{|\mu|}(n)
=
 2\sqrt{n} +  O_\varepsilon(n^{\frac{h}{4}+\varepsilon}),
\end{align}
for every $\varepsilon>0$. 
\end{theorem}
This theorem elaborates on a result of Erd\"{o}s, see \cite{erdossquarefree} and \cite{zaharescusquarefree}, where an asymptotic for $\log\mathfrak{p}_{|\mu|}(n)$ was established. Moreover, we will state in Lemma~\ref{lem:asympt_p_mu} a more precise version of Theorem~\ref{thm:partitionTheorem} using the corresponding saddle point solution. 

\subsection{Discussion on the technique}
We have chosen to present the partitions with weights associated to $|\mu|$ for two main reasons. 

First the function $|\mu|$ is obviously multiplicative and a tempting course of action would have been to use the results in \cite{MV77}, which are extremely general bounds for exponential sums associated to multiplicative functions with certain growth conditions. Our rationale is that there might be other situations where the exponential sum might be twisted, for example by a function $\beta$, say, that is either additive or neither multiplicative nor additive. In those instances, the strategy put forward is to examine the arithmetic properties of $\beta$ to understand if it is the combination of other well understood arithmetic functions such as $\mu, \tau, \mathbf{1}_\mathbb{P}, |\mu|, \log$ etc. There are a handful of fundamental estimates, such as Vinogradov's inequality, from which bounds for exponential sums may be achieved by suitable combinations. In general the results of this strategy are more than good enough to treat a wide range of minor arcs, however more tailored approaches could yield tighter bounds. In addition, our results do not impose growth conditions nor multiplicativity on $\beta$. If one were interested in minor arcs associated to $\mu_\mathbb{P}$ (neither multiplicative nor additive) or $\omega$ (additive), then Theorems \ref{thm:muPri} and \ref{thm:Pri1}, respectively, would be the key results to use. 

The idea introduced in this paper can be applied to other interesting exponential sums. 
For example, let $\Omega$ be the total number of prime factors of $n$, and let $P(s)$ be the prime zeta function. 
Since 
$
\sum_{n=1}^\infty \Omega(n)n^{-s} = \zeta(s) \sum_{k \ge 1} P(ks),
$
for $\real(s)>1$, we can write $\Omega(n) = \sum_{k \ge 1}(\mathbf{1} * \mathbf{1}_{\mathbb{P}^{*k}})(n)$.
We may then apply the ideas from Theorem~\ref{thm:Pri1} to this decomposition to obtain a bound on $\sum_{n \le X} \Omega(n) \e(\alpha n)$, see also \cite{gafniRoblesVaughan}. 
In this way, it is likely that the technique will allow bounding exponential sums associated with a wide range of additive functions. 

Other more exotic examples involve the following functions \cite[$\mathsection$1.2]{Titchmarsh1986}. Let $f_k(n)$ denote the number of representations of $n$ as a product of $k$ factors, each greater than unity when $n>1$, the order of the factors being essential. The Dirichlet series for $f_k$ is $\sum_{n=2}^\infty f_k(n) n^{-s} = (\zeta(s)-1)^k$ for $\real(s) > 1$. Using the binomial expansion, we see that we can write $f_k$ as a convolution of $\tau$ functions, and therefore we may bound the exponential sum twisted by $f_k$. The same can be said about $g(n)$, the number of representations of $n$ as a product of factors greater than unity, representations with factors in a different order being considered as distinct. If we let $f(1)=1$, then $f(n) = \sum_{k=1}^\infty f_k(n)$ and therefore the Dirichlet series is $\sum_{n=1}^\infty f(n) n^{-s} = (2-\zeta(s))^{-1}$ for $\real(s) > \alpha_0$ where $\zeta(\alpha_0) = 2$. In this case, a fractional binomial expansion would also yield combinatorial sums involving $\mu$, and hence, in principle, we could also bound the exponential sum associated to $f$. 

Before discussing the second reason, we also mention that our technology also allows us to bound sums with $\e(\alpha P(n))$ where $P$ is a suitable polynomial. In particular, we could have chosen to study the partitions $\mathfrak{p}_{|\mu|,2}$ associated to 
\begin{align} \label{eq:generatingPartitionsn2}
\Psi_2(z) = \sum_{n=0}^\infty \mathfrak{p}_{|\mu|,2}(n)z^n = \prod_{n=1}^\infty (1-z^{n^2})^{-|\mu(n)|}.
\end{align}
The function $\mathfrak{p}_{|\mu|,2}(n)$ can be interpreted combinatorially as the number of partitions of $n$, where each part is a square of a square free integer, e.g. $4$ is allowed but $16$ is not.
The critical element for the extraction of these partitions would then be Theorem \ref{thm:exponential sum for mu squared squared}.

The second reason is that unlike $\mu$, the function $|\mu|$ only takes values $0$ and $1$. Therefore, $|\mu|$ does not exhibit the same oscillations as $\mu$. As a result, in \cite{BRZ23} the authors achieved $\log \mathfrak{p}_{\mu} \ll \sqrt{n} (\log n)^{-B}$ for any fixed $B>0$ and all $n \ge 2$. In our present case, we are now able to obtain partition asymptotics  of $\log \mathfrak{p}_{|\mu|}$ with a main term as well as an error term as seen in Theorem \ref{thm:partitionTheorem}.

\subsection{Structure of the paper}
The paper is organized as follows. In Section \ref{sec:sec2} we prove Theorem \ref{thm:mainresult}. This will require some preliminary lemmas that we will establish along the way. The special cases $r=2$ and $r=3$ (Theorem \ref{thm:S2S3}) will be treated in Section \ref{sec:sec3}. The bounds on exponential sums from Theorem \ref{thm:mumu}, Theorem \ref{thm:muPri}, Theorem \ref{thm:Pri1} and Theorem \ref{thm:exponential sum for mu squared} will be studied in Section \ref{sec:sec4}. Moreover, in Section \ref{sec:sec5} we will present the proof of Theorem \ref{thm:exponential sum for mu squared squared} along with some additional results, such as Weyl's bound and an analogue of the Heath-Brown identity for $\mathbf{1}_\mathbb{P}$. These results could become useful in the context of bounding new exponential sums associated to arithmetic functions or improving existing ones. Lastly, Section \ref{sec:partitionproof} will be devoted to the proof of Theorem \ref{thm:partitionTheorem} as well as the connection to the zeros of the Riemann zeta-function. We conclude with some additional ideas for future work in Section \ref{sec:sec6}; in particular we elaborate on how our results could be employed in the context of Goldbach-Vinogradov ternary problems.

\subsection{Notation}
Throughout the paper, the expressions $f(X)=O(g(X))$, $f(X) \ll g(X)$, and $g(X) \gg f(X)$ are equivalent to the statement that $|f(X)| \le (\ge) C|g(X)|$ for all sufficiently large $X$, where $C>0$ is an absolute constant. A subscript of the form $\ll_{\alpha}$ means the implied constant may depend on the parameter $\alpha$. The notation $f = o(g)$ as $x\to a$  means that $\lim_{x\to a} f(x)/g(x) = 0$ and $f\sim g$ as $x\to a$ denotes $\lim_{x\to a} f(x)/g(x)= 1$. Dyadic sums are represented by $\sum_{n \sim N} f(n) = \sum_{N < n \leqslant 2N} f(n)$. The divisor function is denoted by $\tau(n)$. The $k$-fold divisor function $\tau_k(n)$ is defined by the coefficients of the Dirichlet series 
$\zeta^k(s) = \sum_{n=1}^{\infty}\frac{\tau_k(n)}{n^s}$ for $\real(s)>1$. The notation $f^{*r}$ indicates that the arithmetic function $f$ is convolved with itself $r$ times. The digamma function is denoted by $\Psi(x) = \frac{\Gamma'}{\Gamma}(x)$. The Greek character $\rho$ is reserved for the radius of the circle method, whereas the non-trivial zeros of zeta will be denoted by $\varpi$. The Latin character $p$ will always denote a prime, whereas $\mathfrak{p}_A(n)$ denotes the number of partitions of $n$ with respect to some weight $A$. Finally, throughout the paper, we use shall use the convention that $\varepsilon$ denotes and arbitrarily small positive quantity that may not be the same at each occurrence.

\section{Proof of Theorem \ref{thm:mainresult}} \label{sec:sec2}

We begin by proving the $r=1$ case separately in the following self-contained result.
\begin{lemma}
\label{thm:general_vinogradov}
Let $\alpha \in \R$,  $a\in\Z$, $q\in\N$ and $\Upsilon>0$ such that $|\alpha-\frac{a}{q}|\leq \frac{\Upsilon}{q^2}$ with $(a,q)=1$. For any $X\geq 2$, one has
\begin{align*}
S_1(\alpha,X) 
&
\ll 
\bigg(\frac{X\max\{1,\sqrt{2\Upsilon}\}}{\sqrt{q}}+X^{\frac{4}{5}}+\sqrt{X}\sqrt{q}\bigg)(\log X)^3.
\end{align*}
The same bound holds for $\widetilde{S}_1(\alpha,X)$ with $(\log X)^3$ replaced by $(\log X)^4$.
\end{lemma}	

\begin{proof}
We only give the proof for $S_1(\alpha,X)$ since the proof for $\widetilde{S}_1(\alpha,X)$ is almost identical.
Dirichlet's approximation theorem asserts that for any real $\alpha$ and for any  $M\geq 1$, there exists a rational number $a_1/q_1$, where $1 \leq q_1 \leq M$, and $(a_1,q_1)=1$, such that $\lvert \alpha-a_1/q_1 \rvert \leq 1/(q_1M)\leq1/q_1^2$. 
Choosing $M=2q$ implies there exists $1\leq q_1\leq 2q$ and $a\in\Z$ with $(a_1,q_1)=1$ and
\begin{align}
\left\lvert \alpha-\frac{a_1}{q_1} \right\rvert \leq \frac{1}{2qq_1}\leq\frac{1}{q_1^2}.\label{Dirichlet theorem}
\end{align}
Inserting $a_1/q_1$ into Vinogradov's bound, i.e Theorem~\ref{eq:originalVinogradov} (or Theorem~\ref{thm:general_vinogradov_old} with $\Upsilon=1$), yields 
\begin{align}
S_1(\alpha,X)
=
\left(\frac{X}{\sqrt{q_1}}+X^{\frac{4}{5}}+\sqrt{X}\sqrt{q_1}\right)(\log X)^3.
\label{eq:vinogradov bound}
\end{align}
We now can have two cases, $a/q=a_1/q_1$ and $a/q\neq a_1/q_1$.
If $a/q=a_1/q_1$, we conclude that $a_1=a$ and $q_1=q$  since both are in the reduced form.
This implies that 
\begin{align}
S_1(\alpha,X)
=
\left(\frac{X}{\sqrt{q}}+X^{\frac{4}{5}}+\sqrt{X}\sqrt{q}\right)(\log X)^3.
\label{eq:case 1 general bound}
\end{align}
Otherwise, if $a/q\neq a_1/q_1$, then using \eqref{Dirichlet theorem}, we have
\begin{align}
\frac{1}{qq_1}&\leq\left\lvert\frac{a_1}{q_1}-\frac{a}{q}\right\rvert\leq\left\lvert\frac{a_1}{q_1}-\alpha\right\rvert+\left\lvert\alpha-\frac{a}{q}\right\rvert\leq \frac{1}{2qq_1}+\frac{\Upsilon}{q^2}.\label{inequality}
\end{align}
Subtract $1/(2qq_1)$ from both sides of \eqref{inequality}, we have $1/(2qq_1)\leq\Upsilon/q^2$, which is equivalent to 
\begin{align*}
\frac{1}{\sqrt{q_1}}\leq\frac{\sqrt{2\Upsilon}}{\sqrt{q}}.
\end{align*}
Inserting this into \eqref{eq:vinogradov bound} implies
\begin{align}
S_1(\alpha,X) 
\ll 
\bigg(\frac{\sqrt{2\Upsilon}X}{\sqrt{q}} +X^{\frac{4}{5}} +X^{\frac{1}{2}}\sqrt{2q}\bigg)(\log X)^4.
\label{eq:case 2 general bound}
\end{align}   
Therefore, combining \eqref{eq:case 1 general bound} and \eqref{eq:case 2 general bound}, we get our desired result.
\end{proof}

Before stating the bounds we wish to show on $S_r(\alpha,X)$ and $\widetilde{S}_r(\alpha,X)$ we need to introduce two key auxiliary results that will be critical in the subsequent proofs.

\begin{lemma}[Generalized bilinear estimate]
\label{lem:ik}
Let $\alpha\in\R$, $a\in\Z$, $q\in\N$ and $\Upsilon>0$ such that $|\alpha-\frac{a}{q}|\leq \frac{\Upsilon}{q^2}$ with $(a,q)=1$.
Further, let $(\xi_m)_{m\in\N}$ and $(\eta_m)_{m\in\N}$ be two real sequences with $|\xi_m|\leq 1$ and $|\eta_n|\leq 1$. 
Then one has for all $M,N>0$ that
\begin{align*}
\sum_{\substack{mn\leq X\\ m>M , n>N}}
\xi_m  \eta_n  \e(\alpha mn)
\ll
\left(\frac{X\max\{1,\Upsilon\}}{q}+\frac{X}{M}+\frac{X}{N}+q\right)^{\frac{1}{2}}X^{\frac{1}{2}}(\log X)^2.
\end{align*}
\end{lemma}
The proof of Lemma~\ref{lem:ik} for $\Upsilon=1$ can be found in \cite[Lemma~13.8]{IwKo04} and the extension to $\Upsilon >0$ follows from the same principles that appeared in the proof of Lemma~\ref{thm:general_vinogradov}. 

Furthermore, we shall introduce in the proof of Theorem~\ref{thm:mainresult} some helpful parameters $M$ and $N$.
After deducing upper bounds for the occurring terms as a function of $M$ and $N$, we have to choose these parameters
in a suitable way to minimize the resulting upper bounds. To accomplish this optimization, we use the following useful lemma. 
\begin{lemma}[{\cite[Lemma~5.1]{semiprimes}}] 
\label{lem:min_max_for_minor}
Let $F$, $G_0$, $G_1$ and $G_2$ be continuous, real valued functions on $\R_+$ such that $F$ is strictly decreasing and all $G_i$ are increasing.
Further, suppose that for $i=0,1,2$
\begin{align}
\lim_{x\to\infty} F(x) = \lim_{x\to 0} G_i(x) = 0
\ \text{ and } \
\lim_{x\to0} F(x) = \lim_{x\to\infty } G_i(x) = \infty.
\end{align} 
Set $G(x):= \max\{G_0(x),G_1(x),G_2(x)\}$ and $H(x):=\max\{F(x),G(x)\}$. 
We then have
\begin{align}
\min_{x\in(0,\infty)} H(x) 
=
F(\min\{U_0,U_1,U_2\})
=
G(\min\{U_0,U_1,U_2\}),
\label{eq:lem:min_max_for_minor}
\end{align}
where $U_i$ is the solution of the equation $F(U_i) = G_i(U_i)$ for $i=0,1,2$.
\end{lemma}

The parameter $x$ runs in \eqref{eq:lem:min_max_for_minor} over all values in $\R_+$. However, we have to impose in the proof of Theorem~\ref{thm:mainresult} that $x\leq X$.
Equation \eqref{eq:lem:min_max_for_minor} is also correct in this situation as long as $\min\{U_0,U_1,U_2\} \leq X$ because
$G(\min\{U_0,U_1,U_2\}) \leq G(X)$.

Equipped with these tools, we may now prove Theorem~\ref{thm:mainresult}. We do this in a slightly more general setting.
For this, let $f$ be an arithmetic function. 
Define for $r\in\N$ and $\alpha\in\R$
\begin{align}
S_{r,f}(\alpha,X)
:=
\sum_{n\leq X} f^{*r}(n) \e(n\alpha).
\label{eq:def_S_rf}
\end{align}
Assume there exists is an $\eta\geq 0$ such that for all $r\in\N$ we have
\begin{align}
f^{*r}(X)\ll (\log X)^{r\eta}
\quad \textnormal{and} \quad
\sum_{n\leq X} |f|^{*r}(n)
\ll_r
X^{} (\log X)^{\eta r}
\label{eq:prop_of_f}
\end{align}	
for $X\geq 2$. 
Abel summation immediately implies that we have for all $0\leq y<1$
\begin{align}
\sum_{n\leq X} |f|^{*r}(n) n^{-y}
\ll_r
X^{1-y} (\log X)^{\eta r}. 
\label{eq:prop_of_f_abel}
\end{align}
Also assume that for all $\alpha\in\R$, $a\in\Z$, $q\in\N$ with $|\alpha-a/q|\leq q^{-2}$ and $(a,q)=1$, 
we have
\begin{align}
S_{1,f}(\alpha,X)
\ll 
(\log X)^{3+r\eta}(Xq^{-\frac{1}{2}}+X^{\frac{4}{5}}+X^{\frac{1}{2}}q^{\frac{1}{2}}).
\label{eq:Vinogradov_f}
\end{align}
\begin{theorem}
\label{thm:minorarclemmaprimorial_f}	
Let $f$ be as above. 
Let $\alpha\in\R$, $a\in\Z$, $q\in\N$ and $\Upsilon>0$ such that 
\begin{align}
\left|\alpha-\frac{a}{q}\right|\leq \frac{\Upsilon}{q^2}, \quad (a,q)=1.  
\label{eq:approx_alpha_for_main2}
\end{align}
We then have for any $X\geq 2$ and $r\in\N$
\begin{align} 
S_{r,f}(\alpha,X) \label{eq:s3finalprimorial_f}
&\ll_r 
(\log X)^{3+r\eta} (X q^{-\frac{1}{2r}} \max\{1,\Upsilon^{\frac{1}{2r}}\}+ X^{\frac{2+2r}{3+2r}} + X^{\frac{2r-1}{2r}}q^{\frac{1}{2r}}).
\end{align}
\end{theorem}
We now give a more compact formulation of \eqref{eq:s3finalprimorial_f}.
We define for $r\in\N$
\begin{align}
\begin{array}{lll}
\beta_0(r) =1,& \quad \beta_1(r) =\frac{2+2r}{3+2r},& \quad \beta_2(r)= \frac{2r-1}{2r},\\
\gamma_0(r)=\frac{1}{2r},& \quad \gamma_1(r)=0,& \quad \gamma_2(r)=0,\\
\delta_0(r)= -\frac{1}{2r},& \quad \delta_1(r) =0,& \quad \delta_2(r) =\frac{1}{2r}.
\end{array}
\end{align}
With these, we can rewrite \eqref{eq:s3finalprimorial_f} for $\Upsilon\geq1$ as
\begin{align}
S_{r,f}(\alpha,X)
&
\ll 
(\log X)^{3+r\eta}\sum_{j=0}^2 X^{\beta_j(r)} \Upsilon^{\gamma_j(r)}  q^{\delta_j(r)}.
\label{eq:s3finalprimorial_3}
\end{align}
The advantage of \eqref{eq:s3finalprimorial_3} over \eqref{eq:s3finalprimorial_f} is that it is easier to handle within complex calculations. In particular, one can often handle all three summands in \eqref{eq:s3finalprimorial_f} at once, rather than having to look at each one separately. 
This is the case, for example, in the proof of Theorem~\ref{thm:minorarclemmaprimorial_f}.
Further, we will use in the proof that each $\beta_{j}(r)$ is an increasing function in $r$ and that each $\beta_{j}(r)$ fulfils the recurrence relation 
\begin{align}
\beta_{j}(r+1)= \frac{2+2\gamma_j-\beta_{j}(r)}{3+2\gamma_j-2\beta_{j}(r)}.
\label{eq:reccurance_beta}
\end{align}

\begin{proof}[Proof of Theorem~\textnormal{\ref{thm:minorarclemmaprimorial_f}}]
For $q\geq X$, the bound in Theorem~\ref{thm:minorarclemmaprimorial_f} is larger than 
the trivial bound we get from \eqref{eq:prop_of_f}.	
Thus we can assume that $q\leq X$.
We prove Theorem~\ref{thm:minorarclemmaprimorial_f} by induction over $r$.  
The case $r=1$ and $\Upsilon=1$ is true by \eqref{eq:Vinogradov_f}.
The case  $r=1$ and $\Upsilon>0$ follows with  the same argument as in the proof of Lemma~\ref{thm:general_vinogradov}.
Thus, Theorem~\ref{thm:minorarclemmaprimorial_f} holds for $r=1$.


Thus we assume that Theorem~\ref{thm:minorarclemmaprimorial_f} holds for all $s\leq r-1$ for some $r\geq2$. 
We now need to show that the theorem also holds for $r$. 
It is sufficient to consider the case $\Upsilon=1$, since we can use the same argument as in the proof of Theorem~\ref{thm:general_vinogradov} to derive the case $\Upsilon>0$.
Note that we require within the proof that the theorem hold for $s\leq r-1$ with $\Upsilon\in\N$ and we thus cannot drop the $\Upsilon$ term from the theorem.

Inserting the definition of the Dirichlet convolution, we see that
\begin{align}
	S_{r,f}(\alpha,X)
	=
	\sum_{u\leq X} f^{*r}(u) \e(u\alpha)
	=
	\sum_{mn\leq X} f^{*(r-1)}(m) f^{}(n)\e(nm\alpha).
%
	\label{eq:def_S_rf2}
\end{align}

We now split $S_{r,f}(\alpha,X)$ into four pieces and estimate them separately.
For this, let $M,N\geq 1$ with $MN\leq X$ arbitrary. 
We will determine them below to optimize our estimate.
We write
\begin{align}
	S_{r,f}(\alpha,X)
	= 
	S_{r,1}(\alpha,X) + S_{r,2}(\alpha,X) +S_{r,3}(\alpha,X) - S_{r,4}(\alpha,X)
	\label{eq:splitting:sum_in_four}
\end{align}
with the $S_{r,j}$ terms given by
\begin{align*}
	S_{r,1}(\alpha,X)
	&=
	\sum_{\substack{mn\leq X\\m> M, n> N}} f^{*(r-1)}(m) f^{}(n)\e(nm\alpha),&
	\quad 
	S_{r,2}(\alpha,X) 
	&=	
	\sum_{\substack{mn\leq X\\ n\leq N}} f^{*(r-1)}(m) f^{}(n)\e(nm\alpha),\nonumber\\
	S_{r,3}(\alpha,X)
	&=
	\sum_{\substack{mn\leq X\\ m\leq M}} f^{*(r-1)}(m) f^{}(n)\e(nm\alpha),
	&\quad
	S_{r,4}(\alpha,X)
	&=
	\sum_{\substack{mn\leq X\\m\leq M, n\leq N}}f^{*(r-1)}(m) f^{}(n)\e(nm\alpha).
\end{align*}
%
	%
We begin by giving an upper bound for $S_{r,1}(\alpha,X)$ with Lemma~\ref{lem:ik}.
To apply Lemma~\ref{lem:ik}, we define
\begin{align}
	\xi_m 
	:=
	f^{*(r-1)}(m)\one_{m\leq X} 
	\quad \textnormal{and} \quad
	\eta_n
	:=
	f(n)
	\one_{n\leq X},
\end{align}
where $\one_{m\leq X}$ and $\one_{n\leq X}$ are indicator functions.
Equation \eqref{eq:prop_of_f} now implies that 
\[
\xi_m \ll (\log X)^{\eta (r-1)} \quad \textnormal{and} \quad \eta_n \ll (\log X)^{\eta}.
\]
Combining this with Lemma~\ref{lem:ik} gives 
\begin{align}
	S_{r,1}(\alpha,X)
	&\ll
	\left(\frac{X}{M}+\frac{X}{N}+\frac{X}{q}+q\right)^{\frac{1}{2}}X^{\frac{1}{2}}(\log X)^{2+\eta r}\nonumber\\
	&\ll
	X(\log X)^{2+\eta r} M^{-\frac{1}{2}} + 
	X(\log X)^{2+\eta r} N^{-\frac{1}{2}} + 
	X(\log X)^{2+\eta r} q^{-\frac{1}{2}} + 
	X^{\frac{1}{2}}(\log X)^{2+\eta r} q^{\frac{1}{2}}.
	\label{eq:sum_pr_1}
\end{align}
The third summand in \eqref{eq:sum_pr_1} is smaller than the first summand in \eqref{eq:s3finalprimorial_f}.
Note that $q\leq X$ implies $X^{\frac{1}{2}}q^{\frac{1}{2}}\leq X^{u}q^{1-u}$ for all $u\geq \frac{1}{2}$.
Thus we get that the fourth summand in \eqref{eq:sum_pr_1} is smaller than the third summand in \eqref{eq:s3finalprimorial_f}.
Hence those two terms have the required order.
It remains to take care of the first two summands in \eqref{eq:sum_pr_1}.
To do this, we combine them with $S_{r,2}(\alpha,X)$ and $S_{r,3}(\alpha,X)$.
More precisely, we set 
\begin{align}
	S_{r,2,N}(\alpha,X)
	&:=
	S_{r,2}(\alpha,X) + X(\log X)^{2+\eta r} N^{-\frac{1}{2}}
	\ \text{ and } \label{def:Sr2N}\\
	S_{r,3,M}(\alpha,X)
	&:=
	S_{r,3}(\alpha,X) + X(\log X)^{2+\eta r} M^{-\frac{1}{2}}.
	\label{def:Sr3M}
\end{align}
Now $S_{r,2,N}(\alpha,X)$ does not depend on $M$ and $S_{r,3,M}(\alpha,X)$ does not depend on $N$.
Thus we can optimize $S_{r,2,N}(\alpha,X)$ and $S_{r,3,M}(\alpha,X)$ separately with respect to $N$ and $M$.
The final step is to insert the chosen $M$, $N$ into $S_{r,4}(\alpha,X)$ and to check that this sum has also the required order.
	
We can do the computations for $S_{r,2,N}(\alpha,X)$ and $S_{r,3,M}(\alpha,X)$ in one sweep.
For this, let $1\leq s\leq r-1$, $U\leq X$ and set
\begin{align}
\Sigma_{r,s}(\alpha,X)
&:=
\sum_{\substack{uv\leq X\\ u\leq U}}
f^{*s}(u)f^{*(r-s)}(v)\e(uv\alpha),
\\
\Sigma_{r,s,U}(\alpha,X) 
&:=	
X(\log X)^{2+\eta r} U^{-\frac{1}{2}}
+
\Sigma_{r,s}(\alpha,X).
\label{eq:Sigma_r,s_def}
\end{align}
%
%
Using the definition of $S_{r,2,N}(\alpha,X)$ and $S_{r,3,M}(\alpha,X)$, we see that
\begin{align}
	\Sigma_{r,1,N}(\alpha,X) = S_{r,2,N}(\alpha,X)
	\quad \text{and} \quad
	\Sigma_{r,r-1,M}(\alpha,X) = S_{r,3,M}(\alpha,X).
	\label{eq:S_sigma}
\end{align}
Then we get
\begin{align}
	\Sigma_{r,s}(\alpha,X)
	&\leq
	\sum_{\substack{u\leq U}} |f^{*s}(u)|
	\bigg|\sum_{v\leq X/u }
	f^{*(r-s)}(v) \e(uv\alpha ) \bigg|.
\end{align}
Since $r-s\leq r-1$, we can apply the induction hypothesis to the inner sum with 
$\alpha' = u\alpha$.
To do this, we have to find suitable $a', q'$ and $\Upsilon$ as in \eqref{eq:approx_alpha_for_main2} such that $|\alpha-a'/q'|\leq \Upsilon/(q')^2$.
Multiplying \eqref{eq:approx_alpha_for_main2} by $u$ leads to
\begin{align}
	\left|\alpha'-\frac{ua}{q}\right|\leq \frac{u}{q^2}.
	\label{eq:alpha_prime1_0}
\end{align} 
By assumption, we have $(a,q)=1$, however $q$ and $u$ do not need to be coprime.
In order to apply the induction hypothesis, we set
\begin{align}
	\alpha' = u\alpha ,  	\quad
	a' =\frac{ua}{\left(u,q\right)},  \quad
	q' =\frac{q}{\left(u,q\right)} \quad \text{and} \quad
	\Upsilon =  \frac{u}{\left(u,q\right)^2},	
	\label{eq:alpha_prime1}
\end{align}
Inserting these expressions into \eqref{eq:alpha_prime1_0} leads to 
\begin{align}
	\left|\alpha'-\frac{a'}{q'}\right|\leq  \frac{\Upsilon }{(q')^2}.
\end{align} 
Now $(a',q')=1$ and the required assumptions are fulfilled.
Inserting the induction hypothesis for $r-s$ in the form of \eqref{eq:s3finalprimorial_3} and using that $\log(X/u) \leq \log X $ gives
\begin{align}
	\Sigma_{r,s}(\alpha,X)
	\ll \,&
	(\log X)^{3+\eta(r-s)}
	\sum_{u \leq U}
	|f^{*s}(u)|
	\sum_{j=0}^2 (X/u)^{\beta_j(r-s)} \Upsilon^{\gamma_j(r-s)} (q')^{\delta_j(r-s)}.
	\label{eq:upper:Sigma}
\end{align}
Observe that $\delta_j(r-s)+\gamma_j(r-s)\geq 0$.
Thus we get
\begin{align*}
	\sum_{j=0}^2 \bigg(\frac{X}{u}\bigg)^{\beta_j(r-s)} \Upsilon^{\gamma_j(r-s)} (q')^{\delta_j(r-s)}
	&=
	\sum_{j=0}^2 X^{\beta_j(r-s)}  q^{\delta_j(r-s)}u^{\gamma_j(r-s)-\beta_j(r-s)} ((u,q))^{-\delta_j(r-s)-2\gamma_j(r-s)}\\
	&\leq
	\sum_{j=0}^2 X^{\beta_j(r-s)}  q^{\delta_j(r-s)}u^{\gamma_j(r-s)-\beta_j(r-s)}, 
\end{align*}
where, we recall, $(u,q)=\gcd(u,q)$. Therefore we arrive at
\begin{align}
	\Sigma_{r,s}(\alpha,X)	
	\ll \,&
	(\log X)^{3+(r-s)\eta}\sum_{j=0}^2 X^{\beta_j(r-s)}  q^{\delta_j(r-s)}
	\sum_{u \leq U} |f^{*s}(u)|
	u^{\gamma_j(r-s)-\beta_j(r-s)}.
	\label{eq:sum_pr} 
\end{align}
Further, we have $-1<\gamma_j(r-s)-\beta_j(r-s)\leq 0$.  
Thus, we get with \eqref{eq:prop_of_f_abel} that
\begin{align}
\sum_{m\leq U} |f(u)|^{*s} u^{\gamma_j(r-s)-\beta_j(r-s)} 
&\ll	
U^{1+\gamma_j(r-s)-\beta_j(r-s)} (\log X)^{\eta s}. 
\label{eq:bound_sum_f*s_u^alpha}
\end{align}
%
Inserting this computation into \eqref{eq:sum_pr} and using that $U\leq X$ yields
\begin{align}
	\Sigma_{r,s}(\alpha,X)
	\ll
	(\log X)^{3+r\eta}
	\bigg(\sum_{j=0}^2 X^{\beta_j(r-s)}  q^{\delta_j(r-s)} U^{1+\gamma_j(r-s)-\beta_j(r-s)}\bigg).
\end{align}
Inserting this into \eqref{eq:Sigma_r,s_def} leads us to
\begin{align}
	\Sigma_{r,s, U}(\alpha,X)
	\ll
	(\log X)^{3+r\eta}
	\bigg(\frac{XU^{-\frac{1}{2}}}{\log X}
	+
	\sum_{j=0}^2 X^{\beta_j(r-s)}  q^{\delta_j(r-s)}
	U^{1+\gamma_j(r-s)-\beta_j(r-s)}\bigg).
	\label{eq:bound_sigma_rsU}
\end{align}
We now chose $U$ so that it minimizes this bound of $\Sigma_{r,s, U}(\alpha,X)$.
Since $1+\gamma_j(r-s)-\beta_j(r-s)\geq 0$ for all $j$, we can use Lemma~\ref{lem:min_max_for_minor} with 
\begin{align}
	F(U):=\frac{XU^{-\frac{1}{2}}}{\log X}
	\quad \text{and} \quad
	G_j(U):=  X^{\beta_j(r-s)}  q^{\delta_j(r-s)} U^{1+\gamma_j(r-s)-\beta_j(r-s)}.
\end{align}
We thus have to find solutions $U_0$, $U_1$ and $U_2$ of the equation
\begin{align}
	\frac{XU_j^{-\frac{1}{2}}}{\log X}
	=
	X^{\beta_j(r-s)} q^{\delta_j(r-s)}  U_j^{1+\gamma_j(r-s)-\beta_j(r-s)}.	
\end{align}
We immediately get with the recurrence relation of $\beta_j(r)$ in \eqref{eq:reccurance_beta} that 
\begin{align}
	U_j
	&=
	X^{\frac{2-2\beta_j(r-s)}{3+2\gamma_j(r-s)-2\beta_j(r-s)}} q^{-\frac{2\delta_j(r-s)}{3+2\gamma_j(r-s)-2\beta_j(r-s)}}
	(\log X)^{\frac{-2}{3+2\gamma_j(r-s)-2\beta_j(r-s)}}
	\nonumber\\
	&=
	X^{2-2\beta_j(r-s+1)} q^{-\frac{2\delta_j(r-s)}{3+2\gamma_j(r-s)-2\beta_j(r-s)}}
	(\log X)^{\frac{-2}{3+2\gamma_j(r-s)-2\beta_j(r-s)}}.
\end{align}
Inserting the values of $\beta_j(r-s)$, $\gamma_j(r-s)$ and $\delta_j(r-s)$
and using in $U_2$ that $\delta_2(r-s)=1-\beta_2(r-s)$ and the definition of $\beta_2(r-s+1)$ produces
\begin{align}
	U_0
	=
	\frac{q^{\frac{1}{r-s+1}}}{(\log X)^{\varepsilon_0}} , \quad 
	U_1
	=
	\frac{X^{2-2\beta_1(r-s+1)}}{(\log X)^{\varepsilon_1}}, \quad
	U_2
	=
	\frac{X^{2-2\beta_2(r-s+1)}q^{-2(1-\beta_2(r-s+1))}}{(\log X)^{\varepsilon_2}},
	\label{eq:U_0U_1U_2}
\end{align}
where $0\leq \varepsilon_j<1$, and therefore $U= \min\{U_0, U_1, U_2\}$. 
In particular, we get $U\leq X$ as required since 
\[
\beta_1(r)=\frac{2+2r}{3+2r}=1-\frac{1}{3+2r}
\]
and thus $U_1\leq X^{1-\frac{1}{3+2r}}$.

Now Lemma~\ref{lem:min_max_for_minor} and the fact that $0\leq \beta_2(r-s+1)\leq 1$ imply 
\begin{align}
	\Sigma_{r,s, U}(\alpha,X)
	&\ll
	(\log X)^{3+r\eta} 
	\bigg(
	\frac{X}{U^{\frac{1}{2}}\log X}
	\bigg) \nonumber \\
	&\leq 
	(\log X)^{3+r\eta} \bigg(
	\frac{X}{U_0^{\frac{1}{2}}\log X} + \frac{X}{U_1^{\frac{1}{2}}\log X}+ \frac{X}{U_2^{\frac{1}{2}}\log X}
	\bigg)\nonumber\\
	&=
	(\log X)^{3+r\eta} 
	\left(
	X q^{-\frac{1}{2(r-s+1)}}
	+ X^{\beta_1(r-s+1)}
	+ X^{\beta_2(r-s+1)}q^{1-\beta_2(r-s+1)}
	\right).
\end{align}
Inserting $s=1$ and $s=r-1$ then immediately gives the upper bounds for $S_{r,2,N}(\alpha,X)$ and $S_{r,3,M}(\alpha,X)$.
Since $\beta_j(r)$ is increasing in $r$ and $q\leq X$, these bounds have the form \eqref{eq:s3finalprimorial_f}.

Thus it remains to show that $S_{r,4}(\alpha,X)$ is of lower order.
We have 
\begin{align}
	M=\min\{M_0,M_1,M_2\}\leq M_1 
	\ \text{ and } \
	N=\min\{N_0,N_1,N_2\}\leq N_1.
\end{align}
Inserting $s=1$ and $s=r-1$ into \eqref{eq:U_0U_1U_2} and using that $\beta_1(r)=\frac{2+2r}{2+3r}$ gives
\[
M_1=X^{2-2\beta_1(2)}=X^{\frac{2}{7}} \quad \textnormal{and}\quad N_1=X^{2-2\beta_1(r)}=X^{\frac{2}{3+2r}}.
\]
Thus, we have $MN\leq X$.
We get with \eqref{eq:prop_of_f} that
\begin{align}
	S_{r,4}(\alpha,X)
	&\ll
	\bigg|\sum_{\substack{mn\leq X\\n\leq M,n\leq N}} 
	f^{*(r-1)}(m) f^{}(n)\e(nm\alpha)\bigg|
	\ll
\sum_{\substack{mn\leq X\\n\leq M,n\leq N}}  |f^{*(r-1)}(m)||f^{}(n)|\nonumber\\
	&\ll
	\bigg(\sum_{m\leq M}|f|^{*(r-1)}(m)\bigg)\bigg(\sum_{n\leq N} \left|f(n)\right|\bigg)
	\ll
	MN (\log X)^{r\eta}.
\end{align}
	Inserting $M_1$ and $N_1$ gives
	\begin{align}
		S_{r,4}(\alpha,X)
		\ll
		X^{\frac{4}{7(3+2r)}}(\log X)^{r\eta}
		\ll 
		X^{\frac{4}{14}}(\log X)^{r\eta}
		\ll 
		X^{\frac{4}{5}}(\log X)^{r\eta}
		\ll
		X^{\beta_1(r)}(\log X)^{r\eta}.
	\end{align}
	This completes the proof.
 \end{proof}


\section{Proof of Theorem \ref{thm:S2S3}} \label{sec:sec3}

We give in this section the proof of Theorem~\ref{thm:S2S3} which needs to be handled separately from the other theorems in order to show how to obtain tighter exponents.

\subsection{Proof of Theorem \ref{thm:S2S3}}
In other words, we determine the powers of $\log X$ more precisely than in Theorem~\ref{thm:mainresult} for $r=2$ and $r=3$.
The proof however follows a similar blueprint as the proof of Theorem~\ref{thm:mainresult}.
Thus we give only the most relevant steps of $S_{2}(\alpha,X)$, and then highlight the adjustments needed for the other cases.

Let $S_{2,j}(\alpha,X)$ be as in \eqref{eq:splitting:sum_in_four} with $f = \mu$ and $r=2$.
We then have
\begin{align}
	S_{2}(\alpha,X)
	= 
	S_{2,1}(\alpha,X) + S_{2,2}(\alpha,X) +S_{2,3}(\alpha,X) - S_{2,4}(\alpha,X) 
\end{align}
We get with Lemma \ref{lem:ik}
\begin{align}
	S_{2,1}(\alpha,X)
	&\ll
	X(\log X)^2 (M^{-\frac{1}{2}} +N^{-\frac{1}{2}} +q^{-\frac{1}{2}} ) 
	+
	X^{\frac{1}{2}}(\log X)^2 q^{\frac{1}{2}}.
	\label{eq:sum_pr_1_r2 S2,1}
\end{align}
Further, using Lemma~\ref{thm:general_vinogradov} with
\begin{align}
	\alpha' = \alpha p_2,  	\quad
	a' =\frac{ap_2}{\left(p_2,q\right)},  \quad
	q' =\frac{q}{\left(p_2,q\right)} \quad \text{and} \quad
	\Upsilon =  \frac{p_2}{\left(p_2,q\right)^2},
	\label{eq:alpha_prime_r2}
\end{align}
we obtain the following bounds
\begin{align}
	S_{2,2}(\alpha,X)
	\leq 
	\sum_{p_2 \leq N}
	\bigg| \sum_{\substack{p_1\leq X/p_{2}}} \e(\alpha p_1p_2)\bigg|
	&\leq 
	\sum_{p_2 \leq N}
	\bigg(\frac{\sqrt{\Upsilon} X}{\sqrt{q'}p_2}+\left(\frac{X}{p_2}\right)^{\frac{4}{5}}+\left(\frac{Xq'}{p_2}\right)^{\frac{1}{2}}\bigg)(\log (X/p_2))^3
	\nonumber\\
	&\leq 
	(\log X)^3
	\sum_{p_2 \leq N}
	\bigg(\frac{X}{\sqrt{qp_2}}+\left(\frac{X}{p_2}\right)^{\frac{4}{5}}+\left(\frac{Xq}{p_2}\right)^{\frac{1}{2}}\bigg)
	\nonumber\\
	&\leq 
	(\log X)^3
	\bigg(\frac{XN^{\frac{1}{2}}}{\sqrt{q}}+X^{\frac{4}{5}} N^{\frac{1}{5}}+\left(XqN\right)^{\frac{1}{2}}\bigg).
	\label{eq:bound_S22}
\end{align}
Thus, we arrive at
\begin{align}
	S_{2,2,N}(\alpha,X) 
	\ll
	X(\log X)^2N^{-\frac{1}{2}}
	+
	(\log X)^3
	\bigg(\frac{XN^{\frac{1}{2}}}{\sqrt{q}}+X^{\frac{4}{5}} N^{\frac{1}{5}}+\left(XqN\right)^{\frac{1}{2}}\bigg).
	\label{eq:bound_S2N}
\end{align}
To minimize $S_{2,2,N}(\alpha,X)$, we use Lemma~\ref{lem:min_max_for_minor}. 
Solving the corresponding equations gives
\begin{align*}
	N_0
	=
	q^{\frac{1}{2}} \log^{-1} X, \quad
	N_1
	=
	X^{\frac{2}{7}} \log^{-\frac{10}{7}} X, \quad
	N_2
	=
	\frac{\sqrt{X}}{\sqrt{q} (\log X)}
\end{align*}
and $N= \min\{N_0, N_1, N_2\}$. Thus
\begin{align}
	S_{2,2,N}(\alpha,X)
	&\ll
	X(\log X)^2 \bigg(\frac{1}{N_0^{\frac{1}{2}}} + \frac{1}{N_1^{\frac{1}{2}}}+ \frac{1}{N_2^{\frac{1}{2}} }\bigg)\nonumber\\
	&\ll
	X q^{-\frac{1}{4}} \log^{\frac{5}{2}} X
	+
	X^{\frac{6}{7}}  \log^{\frac{19}{7}} X
	+
	X^{\frac{3}{4}} q^{\frac{1}{4}} \log^{\frac{5}{2}} X.
\end{align} 
The computations for $S_{2,3,M}(\alpha,X)$ are identical and give the same bound.
Also, $S_{2,4}(\alpha,X)$ is lower order. This completes the proof of \eqref{eq:newS2}.

We now look at the case $r=3$. We have 
\begin{align}
	S_3(\alpha,X)
	= 
	S_{3,1}(\alpha,X) + S_{3,2}(\alpha,X) +S_{3,3}(\alpha,X) - S_{3,4}(\alpha,X)
\end{align}
with $S_{r,j}(\alpha,X)$ as in \eqref{eq:splitting:sum_in_four} with $f=\mu$ but now $r=3$.
We get with Lemma~\ref{lem:ik}
\begin{align}
	S_{3,1}(\alpha,X)
	&\ll
	X(\log X)^2 (M^{-\frac{1}{2}} +N^{-\frac{1}{2}} +q^{-\frac{1}{2}} ) 
	+
	X^{\frac{1}{2}}(\log X)^2 q^{\frac{1}{2}}.
	\label{eq:sum_pr_1_r2 S3,1}
\end{align}
Using the same technique as in the bound for $S_2$, equation \eqref{eq:newS2}, with 
\begin{align}
	\alpha' = \alpha p_3,  	\quad
	a' =\frac{ap_3}{\left(p_3,q\right)},  \quad
	q' =\frac{q}{\left(p_3,q\right)} \ \text{ and } \
	\Upsilon =  \frac{p_3}{\left(p_3,q\right)^2},
	\label{eq:alpha_prime_r3_2}
\end{align}
we get that
\begin{align*}
	S_{3,2}(\alpha,X)
	&=
	\sum_{\substack{p_1p_2p_3\leq X\\ p_3 \leq N}} \e(\alpha p_1p_2p_3)
	\leq 
	\sum_{p_3 \leq N}
	\bigg| \sum_{\substack{p_1p_2\leq X/p_{3}}} \e(\alpha p_1p_2 p_{3})\bigg|\\
	&\ll
	\sum_{p_3 \leq N} \bigg(\left(\frac{X}{p_3}\right) q^{-\frac{1}{4}}\Upsilon^{\frac{1}{4}} \log^{\frac{5}{2}} X
	+
	\left(\frac{X}{p_3}\right)^{\frac{6}{7}}  \log^{\frac{19}{7}} X
	+
	\left(\frac{X}{p_3}\right)^{\frac{3}{4}} q^{\frac{1}{4}} \log^{\frac{5}{2}} X\bigg)\\
	&\ll
	Xq^{-\frac{1}{4}}\log^{\frac{5}{2}} X\sum_{p_3 \leq N} p_3^{-\frac{3}{4}}  
	+
	X^{\frac{6}{7}}\log^{\frac{19}{7}} X \sum_{p_3 \leq N} 
	p_3^{-\frac{6}{7}}  
	+
	X^{\frac{3}{4}}q^{\frac{1}{4}}\log^{\frac{5}{2}} X\sum_{p_3 \leq N}  p_3^{-\frac{3}{4}}.
\end{align*}
We thus arrive at 
\begin{align*}
	S_{3,2,N}(\alpha,X)
	\ll&\,
	X(\log X)^2 N^{-\frac{1}{2}} 
	+
	Xq^{-\frac{1}{4}}N^{\frac{1}{4}}  \log^{\frac{5}{2}} X
	+
	X^{\frac{6}{7}}N^{\frac{1}{7}}  \log^{\frac{19}{7}} X 
	+
	X^{\frac{3}{4}}q^{\frac{1}{4}}N^{\frac{1}{4}}\log^{\frac{5}{2}} X.   
\end{align*}
To minimize $S_{3,2,N}(\alpha,X)$, we use Lemma~\ref{lem:min_max_for_minor}. 
Solving the corresponding equations gives
\begin{align*}
	N_0
	=
	\frac{q^{\frac{1}{3}}}{(\log X)^{\frac{2}{3}}}, \quad
	N_1
	=
	\frac{X^{\frac{2}{9}}}{(\log X)^{\frac{10}{9}}}, \quad
	N_2
	=
	\frac{X^{\frac{1}{3}}}{q^{\frac{1}{3}}(\log X)^{\frac{2}{3}}}.
\end{align*}
and $N= \min\{N_0, N_1, N_2\}$. Thus
\begin{align}
	S_{3,2,N}(\alpha,X)
	&\ll
	X(\log X)^2 \bigg(\frac{1}{N_0^{\frac{1}{2}}} + \frac{1}{N_1^{\frac{1}{2}}}+ \frac{1}{N_2^{\frac{1}{2}} }\bigg)\nonumber\\
	&\ll
	X q^{-\frac{1}{6}} \log^{\frac{7}{3}} X
	+
	X^{\frac{8}{9}}  \log^{\frac{23}{9}} X
	+
	X^{\frac{5}{6}} q^{\frac{1}{6}} \log^{\frac{7}{3}} X.
\end{align} 
The proof of \eqref{eq:s3finalprimorial_f} in Theorem~\ref{thm:minorarclemmaprimorial_f} shows that $S_{3,3,M}(\alpha,X)$ and $S_{3,4}(\alpha,X)$
are smaller than $S_{3,2,N}(\alpha,X)$ and thus this completes the proof.


\section{Proof of Theorems \ref{thm:mumu}, \ref{thm:muPri}, \ref{thm:Pri1} and \ref{thm:exponential sum for mu squared}} \label{sec:sec4}

We start with $S_{\mu*\mu}$ in \eqref{eq:new_mu_mu}. For this, we have to replace Lemma~\ref{thm:general_vinogradov} by
\begin{lemma}
\label{lem:mu}
Let $\alpha\in\R,a\in \Z,q\in\N$ and $\Upsilon>0$ such that $|\alpha-\frac{a}{q}|\leq \frac{\Upsilon}{q^2}$ with $(a,q)=1$. Then, for every fixed $\varepsilon>0$, we have
\begin{align}
\sum_{n\leq X}\mu(n)\e(n\alpha) 
\ll_\varepsilon
\bigg(\frac{X\max\{1,\sqrt{\Upsilon}\}}{q^{\frac{1}{2}}}+X^{\frac{4}{5}+\varepsilon}+X^{\frac{1}{2}}q^{\frac{1}{2}}\bigg)(\log X)^3.
\end{align}
\end{lemma}
This was shown for $\Upsilon=1$ in \cite[Theorem 1.4]{BRZ23} and \cite[$\mathsection$23]{primesDimitris} and the extension to $\Upsilon>0$ follows from the blueprints presented in this paper. Inserting this into the above computation, \eqref{eq:bound_S2N} is replaced by
\begin{align}
	S_{\mu*\mu,2,N}(\alpha,X) 
	\ll
	X(\log X)^2N^{-\frac{1}{2}}
	+
	(\log X)^3
	\bigg(\frac{XN^{\frac{1}{2}}}{\sqrt{q}}+X^{\frac{4}{5}+\varepsilon} N^{\frac{1}{5}+\varepsilon}+\left(XqN\right)^{\frac{1}{2}}\bigg).
	\label{eq:bound_S2Nmu}
\end{align}
The solutions $N_0$ and $N_2$ are the same and 
\begin{align} \label{eq:auxN1}
N_1
=
X^{\frac{2-10\varepsilon}{7+10\varepsilon}} (\log X)^{-\frac{10}{7}-\varepsilon}.
\end{align}
Inserting \eqref{eq:auxN1} into $X(\log X)^2N^{-\frac{1}{2}}$ leads to the corresponding term in  \eqref{eq:new_mu_mu}.

\subsection{An aside for a weaker generalization of Theorem \ref{thm:mumu}} \label{subsec:aside}
The bound \eqref{eq:new_mu_mu} can also be generalized for $r$-fold convolutions of $\mu$, or $r$-fold convolutions of arithmetic functions that satisfy certain growth conditions. Let $f$ be a bounded arithmetic function and let $S_{r,f}(\alpha,X)$ be as in $\eqref{eq:def_S_rf}$.
Let $k\in \N$ and define
$
\tau_r(n)
:=
\mathbf{1}^{*r}(n)
$
where, we recall, the Dirichlet convolution of the constant function $1$ is performed $r$ times. Observe that for all $r\in\N$ we have
$
	f^{*r}(X)\ll \tau_r(X)
	\label{eq:prop_of_f_mu}
$
for $X\geq 1$. 
%
Also assume that for all $\alpha\in\R$, $a\in\Z$, $q\in\N$ with $|\alpha-a/q|\leq q^{-2}$ and $(a,q)=1$, 
we have
\begin{align} \label{eq:Vinogradov_f_mu_d}
    S_{1,f}(\alpha, X) \ll (\log X)^3 (Xq^{-1/2} + X^{4/5+\varepsilon}+X^{1/2}q^{1/2}).
\end{align}
\begin{theorem}
    \label{thm:minorarclemmaprimorial_f_mu}	
Let $f$ be as above. 
Let $\alpha\in\R$, $a\in\Z$, $q\in\N$ and $\Upsilon>0$ such that $|\alpha-\frac{a}{q}| \leq \frac{\Upsilon}{q^2}$ with $(a,q)=1$. We then have for any $X\geq 2$ and $r\in\N$
\begin{align} 
	S_{r,f}(\alpha,X) 
	&\ll_r 
	(X q^{-\frac{1}{2r}} \max\{1,\Upsilon^{\frac{1}{2r}}\}+ X^{\frac{2+2r}{3+2r}+\varepsilon} + X^{\frac{2r-1}{2r}}q^{\frac{1}{2r}})(\log X)^{\max\{r^2,3\}}.
	\label{eq:s3finalprimorial_f_mu}
\end{align}
\end{theorem}
If we take $f=\mu$, then the generalization follows albeit with a slightly weaker power of $\log X$. We will first need some auxiliary results. The first of which is a result of interest in its own right that can be repurposed for future work.

\begin{lemma}[Generalized Type II estimate]
\label{lem:L2_bound}

Let $g$ and $h$ be two arithmetic functions supported on $[1, y]$ and $[1, z]$, respectively and $X \ge 2$.
Further, let $\alpha \in \R$,  $a\in\Z$, $q\in\N$ and $\Upsilon>0$ such that $|\alpha-\frac{a}{q}|\leq \frac{\Upsilon}{q^2}$ with $(a,q)=1$. We then have as $X\to\infty$
\begin{align}
\sum_{n\leq X} (g*h)(n)\e(\alpha n)
\ll
\left(\frac{yz}{q}\max\{1,\Upsilon\}+y+z+q\right)^{1/2} \left(\log 2q\right)^{1/2} \|g\|_2\|h\|_2,
\end{align}
where $\|\cdot\|_2$ denotes the $L^2$-norm on $\N$.
\end{lemma}
\begin{proof}
The proof for $\Upsilon=1$ can be found in \cite[Theorem 23.5]{primesDimitris}, and the extension to $\Upsilon >0$ follows the same principles as in the proof of Lemma~\ref{thm:general_vinogradov}.
\end{proof} 
Next, we need to make use of a divisor sum estimate.  
In this context, Norton \cite{norton} established the following useful result.
\begin{lemma}
\label{lem:boud_d_r^k}
We have for $r\geq 1$ as $X\to\infty$ 
\begin{align}
\sum_{n\leq X} (\tau_r(n))^2
\ll 
X (\log X)^{r^2-1}.
\label{eq:lem:boud_d_r^k}
\end{align}	
Further, we have for $s<1$ that
\begin{align} 
\sum_{n\leq X} \tau_r(n) n^{-s} 
\ll 
X^{1-s} (\log X)^{r-1}.
\label{eq:lem:boud_d_r}
\end{align}
\end{lemma}
\begin{proof}
	We first look at \eqref{eq:lem:boud_d_r^k}. 
	The case $r=1$ is trivial since $\tau_1=1$. 
	The case $r\geq 2$ can be found in \cite[Theorem 23.6]{primesDimitris}.
	Further, \eqref{eq:lem:boud_d_r} follows by a simple induction.
\end{proof}

With these two results, we may now proceed. 

\begin{proof}[Proof of Theorem~\textnormal{\ref{thm:minorarclemmaprimorial_f_mu}}]
The proof is an enhancement of the proof of Theorem~\ref{thm:minorarclemmaprimorial_f}.
Thus we only state the differences between them with the relevant details.

The case $\Upsilon\neq 1$ can be deduced immediately from the case $\Upsilon=1$ with the same argument as in the proof of Lemma~\ref{thm:general_vinogradov}.
Also, if $q\geq X$ then the bound in \eqref{eq:s3finalprimorial_f_mu} is larger than what we get with the trivial bound combined with \eqref{eq:lem:boud_d_r}.
Thus, we can assume that $q\leq X$ and $\Upsilon=1$.

Equation \eqref{eq:Vinogradov_f_mu_d} implies that the theorem holds for $r=1$.
Thus we can assume that the theorem holds for all $s\leq r-1$ for some $r\geq 2$.
Our task is to show that it also holds for $s=r$.
As in the proof of Theorem~\ref{thm:minorarclemmaprimorial_f}, we write
\begin{align}
	S_{r,f}(\alpha,X)
	= 
	S_{r,1}(\alpha,X) + S_{r,2}(\alpha,X) +S_{r,3}(\alpha,X) - S_{r,4}(\alpha,X)
	\label{eq:splitting:sum_in_four2}
\end{align}
with $S_{r,j}(\alpha,X)$ as in \eqref{eq:splitting:sum_in_four}.
The main difference to the proof of Theorem~\ref{thm:minorarclemmaprimorial_f} is that we have to use a different bound for $S_{r,1}(\alpha,X)$.
More precisely, we replace Lemma~\ref{lem:ik} with Lemma~\ref{lem:L2_bound}.
Recall that for $S_{r,1}$ we have 
\begin{align}
	S_{r,1}(\alpha,X)
	&=
	\sum_{\substack{mn\leq X\\m> M, n> N}} f^{*(r-1)}(m) f^{}(n)\e(nm\alpha).
\end{align}
We cannot apply Lemma~\ref{lem:L2_bound} directly with $g=f^{*(r-1)}$ and $h=f$  since the obtained bound is too large.
To obtain a better bound, we use a dyadic decomposition.
Observe that $m\in[M, X/N]$. 
We now split this interval into dyadic intervals $(2^{j-1},2^j]$ with
$2^j\in[M, X/N]$. Also, if $m\in(2^{j-1},2^j]$ then $n\leq \frac{X}{2^{j-1}}$.
With this in mind, we can write
\begin{align}
	S_{r,1}(\alpha,X)
	=
	\sum_{M<2^j<X/N} \sum_{u\leq X}(g_j*h_j)(u) \e(u\alpha),
\end{align}
where the functions $g_j$ and $h_j$ are given by
\begin{align}
g_j(m) =f^{*(r-1)}(m) \one_{2^{j-1}<m\leq 2^{j}}
\ \text{ and } \
h_j(n) =f(n) \one_{N<n\leq \frac{X}{2^{j-1}}}.
\end{align}
Observe that using Lemma~\ref{lem:boud_d_r^k} leads to
\begin{align}
	\|g_j\|_2^2
	=
	\sum_{m\in\N} g_j^2(m)
	\ll
	\sum_{m=2^{j-1}}^{2^{j}} (\tau_{r-1}(m))^2
	\ll
	2^{j} (\log X)^{(r-1)^2-1}.
\end{align}
Similarly, we get $\|h_j\|_2^2\ll \frac{X}{2^{j-1}}$.
Applying Lemma~\ref{lem:L2_bound} with 
$y=2^j$, $z=\frac{X}{2^{j-1}}$, we obtain
\begin{align}
S_{r,1}(\alpha,X)
=
\sum_{M<2^j<X/N} \sum_{u\leq X}(g_j*h_j)(u)
&\ll
\sum_{M<2^j<X/N} \left(\frac{2X}{q}+2^j+\frac{X}{2^{j-1}}+q\right)^{1/2} X^{\frac{1}{2}}
(\log X)^{(r-1)^2-1}\nonumber\\
&\ll
 \left(\frac{X}{q}+\frac{X}{N}+\frac{X}{M}+q\right)^{1/2}X^{\frac{1}{2}}
(\log X)^{(r-1)^2}.
\end{align}
Thus the bound for $S_{r,1}(\alpha,X)$ is analogous to the bound of $S_{r,1}(\alpha,X)$ that appears the proof of Theorem~\ref{thm:minorarclemmaprimorial_f}, except that we have a different power of $\log X$.

Likewise, the computations for $S_{r,2}(\alpha,X)$ and $S_{r,3}(\alpha,X)$ follow similar guidelines, except that the powers of $\log X$ change and one has to replace $\beta_1(r)=\frac{2+2r}{3+2r}$ by $\beta_1(r) =\frac{2+2r}{3+2r}+\varepsilon$. 
For instance, \eqref{eq:upper:Sigma} has to be replaced with 
\begin{align}
	\Sigma_{r,s}(\alpha,X)
	\ll \,&
	(\log X)^{\max\{(r-s)^2,3\}}
	\sum_{u \leq U}
	|f^{*s}(u)|
	\sum_{j=0}^2 (X/u)^{\beta_j(r-s)} \Upsilon^{\gamma_j(r-s)} (q')^{\delta_j(r-s)}
\end{align}
because of the different $\log$ power in the induction hypothesis.
Also, \eqref{eq:bound_sum_f*s_u^alpha} has to be replaced by
\begin{align}
	\sum_{m\leq U} |f(u)|^{*s} u^{\gamma_j(r-s)-\beta_j(r-s)} 
	&\ll	
	U^{1+\gamma_j(r-s)-\beta_j(r-s)} (\log X)^{s-1}
\end{align}
because we have to use \eqref{eq:lem:boud_d_r}.
Combining both bounds shows that we have to replace \eqref{eq:bound_sigma_rsU} by
\begin{align}
	\Sigma_{r,s, U}(\alpha,X)
	\ll
	(\log X)^{\max\{r^2,3\}}
	\bigg(\frac{XU^{-\frac{1}{2}}}{\log X}
	+
	\sum_{j=0}^2 X^{\beta_j(r-s)}  q^{\delta_j(r-s)}
	U^{1+\gamma_j(r-s)-\beta_j(r-s)}\bigg).
\end{align}
This expression is very close to \eqref{eq:bound_sigma_rsU} and we can use exactly the same computation as in the proof of Theorem~\ref{thm:minorarclemmaprimorial_f} to minimise $\Sigma_{r,s, U}(\alpha,X)$ with respect to $U$. Since these computations are akin, we may omit them.
\end{proof}

\subsection{Proof of Theorem \ref{thm:muPri} and Theorem \ref{thm:Pri1}}
The computation for \eqref{eq:new_mu_1p} is concomitant and a combination of \eqref{eq:bound_S2N} and  \eqref{eq:bound_S2Nmu}. 
It remains to show \eqref{eq:new_1_1p}.
For this we use 
\begin{align*}
	\sum_{m\leq M}\mu(m)\sum_{n\leq X/m}\mu(n)\e(mn\alpha)
	&\ll 
	\sum_{m\leq M} \sum_{n\leq X/m}\lvert \e(mn\alpha)\rvert\\
	&\ll
	\sum_{m\leq M} \min\bigg\{\frac{X}{m}, \frac{1}{||m\alpha||}\bigg\}\ll 
	\bigg(M+\frac{X}{q}+q\bigg)\log X.
\end{align*}
As above, we need to minimize
\begin{align*}
	S_{\mathbf{1}*\mathbf{1}_{\mathbb{P}},2,M}(\alpha, X) 
	:= 
	\frac{X}{M^{\frac{1}{2}}}(\log X)^2+\bigg(\frac{X}{q}+M+q\bigg)\log X
\end{align*}
Applying Lemma~\ref{lem:min_max_for_minor} and solving the resulting equations gives $M = X^{\frac{2}{3}}(\log X)^{\frac{2}{3}}$. Therefore,
\begin{align*}
	S_{\mathbf{1}*\mathbf{1}_{\mathbb{P}},2,M}(\alpha, X) 
	&\ll 
	\bigg(\frac{X}{q}+q\bigg)\log X + X^{\frac{2}{3}}(\log X)^{\frac{5}{3}}.
\end{align*}
Similarly, we get as in the previous cases
\begin{align*}
	S_{\mathbf{1}*\mathbf{1}_{\mathbb{P}},3,M}(\alpha, X) 
	&\ll
	X q^{-\frac{1}{4}} \log^{\frac{5}{2}} X
	+
	X^{\frac{6}{7}}  \log^{\frac{19}{7}} X
	+
	X^{\frac{3}{4}} q^{\frac{1}{4}} \log^{\frac{5}{2}} X.
\end{align*} 
Combing $S_{\mathbf{1}*\mathbf{1}_{\mathbb{P}},2,M}(\alpha, X)$ and $S_{\mathbf{1}*\mathbf{1}_{\mathbb{P}},3,M}(\alpha, X)$ and
using that $q\leq X$ gives \eqref{eq:new_1_1p}. 
Using that $|\mu| \le 1$, the proof for both theorems is identical and the bounds are the same (in fact, it follows for any arithmetic bounded function).

\subsection{Proof of Theorem \ref{thm:exponential sum for mu squared}}
Before proceeding to the proof, it is worth noting that the techniques that we have employed so far ($\Upsilon$-enhanced bounds combined with the hyperbola method) could also have been used to obtain a bound for $S_{|\mu|}(\alpha,X)$. Indeed, from \cite{BRZ23} we have
\begin{align} \label{eq:2328}
	\sum_{n\leq X}\widetilde{\mu}_2(n) \e(n \alpha)
	&\ll
	\bigg(\frac{X}{q^{\frac{1}{4}}}\max\{1,\Upsilon^{\frac{1}{4}}\}+X^{\frac{23}{28}}+X^{\frac{3}{4}} q^{\frac{1}{4}} \bigg)(\log X)^{10},
\end{align}
for $|\alpha-\frac{a}{q}| \le \frac{\Upsilon}{q^2}$ with $\Upsilon>0$ and $(a, q)=1$. Here $\tilde \mu_2$ is generated by the Dirichlet series $\sum_{n=1}^\infty \frac{\tilde  \mu_2(n)}{n^s} = \frac{1}{\zeta(2s)}$ for $\real(s)>\frac{1}{2}$. Combining this with the well-known result
\begin{align}
\sum_{n\leq X}\e(n \alpha)
	&\ll
	\min\bigg\{X, \frac{1}{\|\alpha\|}\bigg\}
\end{align}
yields $S_{|\mu|}(\alpha,X)	= S_{1}(\alpha,X) + S_{2}(\alpha,X) +S_{3}(\alpha,X) - S_{4}(\alpha,X)$ with $S_{j}(\alpha,X)$ similar as in \eqref{eq:splitting:sum_in_four} with $f = |\mu|$ and $r=2$. Applying Lemma~\ref{lem:ik} yields
\begin{align} \label{eq:sum_pr_1_r2 S}
	S_{1}(\alpha,X)
	&\ll
	X(\log X)^{10} (M^{-\frac{1}{2}} +N^{-\frac{1}{2}} +q^{-\frac{1}{2}} ) 
	+
	X^{\frac{1}{2}}(\log X)^{10} q^{\frac{1}{2}},
\end{align}
since $(\log X)^2 \ll (\log X)^{10}$. Next, we set
\begin{align}
	\alpha' = \alpha m,  	\quad
	a' =\frac{am}{\left(m,q\right)},  \quad
	q' =\frac{q}{\left(m,q\right)} \ \text{ and } \
	\Upsilon =  \frac{m}{\left(m,q\right)^2}.
	\label{eq:alpha_prime_r3_22}
\end{align}
Moving on $S_2$ we see that
\begin{align*}
	\sum_{m\leq M} \bigg| \sum_{n\leq X/m} \widetilde{\mu}_2(m)\e(nm\alpha)\bigg| 
	&\ll
	\sum_{n\leq N} \bigg(
	\frac{X}{n}(q')^{-\frac{1}{4}} \Upsilon^{\frac{1}{4}} +\bigg(\frac{X}{n}\bigg)^{\frac{23}{28}}+ \bigg(\frac{X}{n}\bigg)^{\frac{3}{4}}(q')^{\frac{1}{4}}\bigg)(\log X)^{10}\\
		&\ll
	\sum_{n\leq N} 
	\bigg(\frac{X}{q^{\frac{1}{4}}}n^{-\frac{3}{4}}  +X^{\frac{23}{28}}n^{-\frac{23}{28}}+ 	 
	X^\frac{3}{4}n^{-\frac{3}{4}}q^{\frac{1}{4}}\bigg)(\log X)^{10}\\
	&\ll 
	\bigg(\frac{XN^{\frac{1}{4}}}{q^{\frac{1}{4}}}+X^{\frac{23}{28}}N^{\frac{5}{28}}+ 	 
	X^\frac{3}{4}N^{\frac{1}{4}}q^{\frac{1}{4}}
	\bigg)(\log X)^{10}.
\end{align*}
In order to select the optimal values we need to solve the equations $X N^{-\frac{1}{2}}=XN^{\frac{1}{4}}q^{-\frac{1}{4}}, X N^{-\frac{1}{2}}=X^{\frac{23}{28}}N^{\frac{5}{28}}$ and $X N^{-\frac{1}{2}}=X^\frac{3}{4}N^{\frac{1}{4}}q^{\frac{1}{4}}$.
The solutions are $N_0 = q^{\frac{1}{3}}, N_1 = X^{\frac{5}{19}}$ and $N_2=X^{\frac{1}{3}}q^{-\frac{1}{3}}$, respectively.
Inserting these, we obtain $X q^{-\frac{1}{6}} + X^{\frac{33}{38}} + X^{\frac{5}{6}} q^{\frac{1}{6}}$.
Next we study $S_3$. We write
\begin{align}
	\sum_{m\leq M} \widetilde{\mu}_2(m) \sum_{m\leq X/n} \e(nm\alpha)
	&\ll
	\sum_{m\leq M} |\widetilde{\mu}_2(m)|  \bigg| \sum_{m\leq X/n} \e(nm\alpha)\bigg| \nonumber \\
	&\ll
	\sum_{m\leq M}	\min\left\{\frac{X}{n}, \frac{1}{\|n\alpha\|}\right\} \ll
	\bigg(M +q +\frac{X}{q}\bigg) \log X .
\end{align}
We have to solve the equations $X M^{-\frac{1}{2}}=M, X M^{-\frac{1}{2}}=X^{}q^{-1}$ and $X M^{-\frac{1}{2}}=q$.
The solutions are $M_0= X^{\frac{2}{3}}, M_1= q^2$ and $M_2= X^{2}q^{-2}$, respectively.
Inserting these, we get $X^{\frac{2}{3}} + Xq^{-1}+q$.
%
Combining all these results, we finally arrive at
\begin{align} \label{eq:weakerboundmusquared}
\sum_{n \le X} |\mu(n)| \e(\alpha n) &\ll \bigg(
\frac{X}{q^{\frac{1}{6}}} + X^{\frac{33}{38}} + X^{\frac{5}{6}} q^{\frac{1}{6}}
+ 
X^{\frac{2}{3}} + \frac{X}{q}+q
+
\frac{X}{q^{\frac{1}{2}}}  
+
X^{\frac{1}{2}} q^{\frac{1}{2}}\bigg)(\log X)^{10} \nonumber\\
&\ll
\bigg(\frac{X}{q^{\frac{1}{6}}} + X^{\frac{33}{38}} + X^{\frac{5}{6}} q^{\frac{1}{6}}\bigg)(\log X)^{10}.
\end{align}

While this is certainly a useful bound that could be used to deal with the minor arcs arising from partitions weighted by $|\mu|$ as we shall in Section \ref{sec:partitionproof}, we can actually substantially improve this bound to \eqref{eq:new_|mu|}, i.e. the one shown in the statement of Theorem \ref{thm:exponential sum for mu squared}. To accomplish this improvement we shall first to state and prove need an auxiliary result which is both a generalization as well as an improvement of a lemma of Mikawa \cite{mikawa}.

\begin{lemma}\label{generalized Mikawa's Lemma}
    For $2\leq M, J\leq x$ and $k\in\Z^+$, we have that
    \begin{align*}
         G&:=M\sum_{m\sim M}\sum_{j\sim J}\tau_k(j)\min\left\{\frac{x}{m^2j},\frac{1}{||\alpha m^2 j||}\right\} \ll  M^2J(\log x)^k+x^{3/4}\bigg(\frac{x}{q}+\frac{x}{M}+q\bigg)^{1/4}(\log x)^{k+3}.
    \end{align*}
\end{lemma}
\begin{proof}
Denote $\log x$ by $L$. For $H>2$, we have the Fourier expansion
\[
\min(H, ||\theta||^{-1})=\sum_{h\in \Z}w_h \e(\theta h),\]
where $w_h = w_h(H)\ll\min(\log H, \frac{H}{|h|},\frac{H^2}{h^2}).$ Put $H=x(M^2J)^{-1}$. By \eqref{eq:lem:boud_d_r}, we have 
    \[
    \sum_{j\sim J}\tau_k(j)\ll J(\log J)^{k-1}\ll JL^{k-1}.
    \]
    Therefore, if $H\leq 2$, we trivially have that 
    \[
    G\ll M\sum_{m\sim M}\sum_{j\sim J}\tau_k(j)\ll M^2JL^{k-1}.
    \]
Now for $H>2$, we may use the above expansion to obtain
\[
\min\bigg\{H, \frac{1}{||\alpha m^2j||}\bigg\} = O(L)+\sum_{0<|h|<H^2}w_h \e(\alpha m^2jh).
\]
Substituting this into $G$, we see that
\begin{align} \label{GboundwithF}
    G \ll M^2JL^k+M\sum_{0<|h|<H^2}|w_h|\sum_{j\sim J}\tau_k(j) \left\lvert \e(\alpha m^2jh)\right\rvert =m^2JL^k+F. 
\end{align}
By \cite[$\mathsection$4]{mikawa} we can write the bound 
\[
\left\lvert \e(\alpha m^2jh)\right\rvert^2\ll M+\sum_{g\leq 2M}\min\bigg\{M,\frac{1}{||\alpha gjh||}\bigg\}, 
\]
this implies that
\begin{align}
    F^2&\ll M^2\sum_{k\leq H^2}|w_k|\sum_{l\sim J}\tau_k(l)^2\sum_{h\leq H^2}|w_h|\sum_{j\sim J}\left| \e(\alpha m^2jh)\right|^2\notag\\
    &\ll M^2HL\sum_{l\sim J}\tau_{2k-1}(l)\sum_{h\leq H^2}|w_h|\sum_{j\sim J}\left| \e(\alpha m^2jh)\right|^2\notag\\
    &\ll M^2HJL^{2k-1}\bigg\{HJML+\sum_{h\leq H^2}\min\bigg(\log H, \frac{H}{h}\bigg)\sum_{j\sim J}\sum_{g\leq 2M}\min\bigg(M, \frac{1}{||\alpha gjh||}\bigg)\bigg\}\notag\\
    &=xL^{2k-1}\bigg\{\frac{x}{M}L+E\bigg\}, \label{F squared}
\end{align}
where for the $E$ term we have
\[
E\ll L\max_{1\leq T\ll H}\frac{1}{T}\sum_{h\leq 2HT}\sum_{j\sim J}\sum_{g\leq 2M}\min\bigg(M,\frac{1}{||\alpha gjh||}\bigg).
\]
Using again \cite[$\mathsection$4]{mikawa}, $E\ll x^{1/2}(\frac{x}{q}+\frac{x}{M}+q)^{1/2}L^6.$ Therefore, combining this with \eqref{F squared}, we get
\begin{align*}
    F^2\ll xL^{2k-1}\bigg\{\frac{x}{M}L+x^{1/2}\bigg(\frac{x}{q}+\frac{x}{M}+q\bigg)^{1/2}L^6\bigg\} \ll \frac{x^2L^{2k}}{M}+x^{3/2}\bigg(\frac{x}{q}+\frac{x}{M}+q\bigg)^{1/2}L^{2k+5}.
\end{align*}
Thus, combining this with \eqref{GboundwithF}, we see that
\begin{align*}
    G&\ll m^2JL^k+\frac{xL^k}{M^{1/2}}+x^{3/4}\bigg(\frac{x}{q}+\frac{x}{M}+q\bigg)^{1/4}L^{k+3} \ll m^2JL^k+x^{3/4}\bigg(\frac{x}{q}+\frac{x}{M}+q\bigg)^{1/4}L^{k+3},
\end{align*}
which was the last ingredient of the proof.
\end{proof}

We shall also make use of the following well-known result \cite[$\mathsection$13]{IwKo04}.
\begin{lemma}\label{sum of min}
    If $\alpha \in \R$, $a \in Z$ and $q \in \N$ are such that $|\alpha-\frac{a}{q}| \le \frac{1}{q^2}$ with $(a,q)=1$, then
    \[
    \sum_{m\leq M}\min\bigg\{\frac{X}{m},\frac{1}{||m\alpha||}\bigg\}\ll \bigg(M+\frac{X}{q}+q \bigg)\log(2qMX).
    \]
\end{lemma}

We are now ready to proceed with the proof of the improved bound of the exponential sum twisted by $|\mu|$.

\begin{proof}[Proof of Theorem \textnormal{\ref{thm:exponential sum for mu squared}}]
Noticing that $\mu^2(x) = \sum_{ab^2 = x}\mu(b)$ we write
    \begin{align*}
        \sum_{n\leq X}\mu^2(n)\e(n\alpha)& = \sum_{n\leq X}\sum_{ab^2=n}\mu(b)\e(n\alpha)=\sum_{b^2\leq X}\sum_{a\leq X/b^2}\mu(b)\e(ab^2\alpha)\\
        &=\sum_{b\leq X^{1/2}}\mu(b)\sum_{a\leq X/b^2}\e(ab^2\alpha)\ll  \sum_{b\leq X^{1/2}}\min\bigg\{\frac{X}{b^2},\frac{1}{||b^2\alpha||}\bigg\}.
    \end{align*}
We break the range of $b$ into two parts: for large $b$'s, we will use the bound from Lemma \ref{generalized Mikawa's Lemma}, whereas for small $b$'s, we will use Lemma \ref{sum of min}. This means that
\begin{align}
\sum_{b\leq X^{1/2}}\min\bigg\{\frac{X}{b^2},\frac{1}{||b^2\alpha||}\bigg\} = \bigg(\sum_{b\leq B} + \sum_{B<b\leq X^{1/2}}\bigg) \min\bigg\{\frac{X}{b^2},\frac{1}{||b^2\alpha||}\bigg\}.
\end{align}
Applying Lemma \ref{sum of min}, we get
\begin{align}
    \sum_{b\leq B}\min\bigg\{\frac{X}{b^2},\frac{1}{||b^2\alpha||}\bigg\} &\ll \sum_{m\leq B^2}\min\bigg\{\frac{X}{m},\frac{1}{||m\alpha||}\bigg\}\ll \bigg(B^2+ \frac{X}{q}+q \bigg)(\log X).\label{small q small b}
\end{align}
Next,by Lemma \ref{generalized Mikawa's Lemma} dyadically with $k=1$, we have 
\begin{align}
    \sum_{B<b\leq X^{1/2}}\min\bigg\{\frac{X}{b^2},\frac{1}{||b^2\alpha||}\bigg\}&\ll \sum_{B<b\leq X^{1/2}}\sum_{m\leq 2}\min\bigg\{\frac{X}{b^2m},\frac{1}{||b^2m\alpha||}\bigg\}\\
    &\ll\bigg(X^{1/2}+\frac{X}{q^{1/4}B}+\frac{X}{B^{5/4}}+\frac{X^{3/4}q^{1/4}}{B}\bigg)(\log X)^4.\label{large b}
\end{align}
Using Lemma \ref{lem:min_max_for_minor}, combining \eqref{small q small b} and \eqref{large b}, we choose
\[
F(B) = \frac{B^2}{(\log X)^3}, \quad G_1(B) = \frac{X}{q^{1/4}B}, \quad G_2(B)=\frac{X}{B^{5/4}} \quad \textnormal{and} \quad G_3(B)=\frac{X^{3/4}q^{1/4}}{B},
\]
then we get the bound
   \begin{align}
     \frac{B^2}{(\log X)^3} \ll \frac{X^{2/3}}{q^{1/6}\log X}+\frac{X^{8/13}}{(\log X)^{15/13}}+\frac{X^{1/2}q^{1/6}}{\log X}, 
   \end{align} 
and thus we finally arrive at
\begin{align}
 \sum_{n\leq X}|\mu(n)|\e(n\alpha) &\ll \frac{X\log X}{q}+q\log X+X^{\frac{1}{2}}(\log X)^4  +\bigg(\frac{X^{\frac{2}{3}}}{q^{\frac{1}{6}}\log X}+\frac{X^{\frac{8}{13}}}{(\log X)^{\frac{15}{13}}}+\frac{X^{\frac{1}{2}}q^{\frac{1}{6}}}{\log X}\bigg)(\log X)^4\notag\\
 &=\bigg(\frac{X}{q}+q\bigg)\log X+X^{\frac{1}{2}}(\log X)^4+\frac{X^{\frac{2}{3}}}{q^{\frac{1}{6}}}(\log X)^3+X^{\frac{8}{13}}(\log X)^{\frac{37}{13}}+X^{\frac{1}{2}}q^{\frac{1}{6}}(\log X)^3\notag\\
 &\ll \frac{X}{q}\log X+X^{\frac{8}{13}}(\log X)^{\frac{37}{13}}+q\log X,\label{bound 1}
\end{align}
as it was to be shown.
\end{proof}
\begin{remark}
In the statement of Lemma \ref{thm:exponential sum for mu squared}, the bound will be worse than the trivial bound if $q\geq X$. Note that in \eqref{bound 1} if $q\in (0,X^{5/13})$, then $X/q$ dominates; if $q\in (X^{5/13},X^{8/13})$, then $X^{8/13}(\log X)^{\frac{37}{13}}$ dominates; and if $q\in (X^{8/13},X)$, then $q\log X$ dominates.
In \eqref{small q small b}, if we use the Cauchy-Schwarz inequality we see that
\begin{align}
    \sum_{b\leq B}\min\bigg\{\frac{X}{b^2},\frac{1}{||b^2\alpha||}\bigg\} &\ll \sum_{m\leq B^2}\min\bigg\{\frac{X}{m},\frac{1}{||m\alpha||}\bigg\}\notag\\
    &\ll \bigg(X\sum_{m\leq B^2} \frac{1}{m}\bigg)^{\frac{1}{2}} \bigg(\bigg(B^2+\frac{X}{q}+q\bigg)(\log X)\bigg)^{\frac{1}{2}}\notag\\
    &\ll X^{\frac{1}{2}}\bigg(B+\frac{X^{\frac{1}{2}}}{q^{\frac{1}{2}}}+q^{\frac{1}{2}}\bigg)(\log X)^2 = \bigg(X^{\frac{1}{2}}B+\frac{X}{q^{\frac{1}{2}}}+X^{\frac{1}{2}}q^{\frac{1}{2}}\bigg)(\log X)^2. \label{large q small b}
\end{align}
Then again using Lemma \ref{lem:min_max_for_minor}, we will have
\begin{align}
 \sum_{n\leq X}|\mu(n)|\e(n\alpha) &\ll \bigg(\frac{X^{3/4}}{q^{1/8}}+X^{13/18}+X^{5/8}q^{1/8}+\frac{X}{q^{1/2}}+X^{1/2}q^{1/2}\bigg)(\log X)^{4}\notag\\
 &\ll \bigg(\frac{X}{q^{1/2}}+X^{1/2}q^{1/2}\bigg)(\log X)^{4}.\label{bound 2}
\end{align}
Now, in \eqref{bound 2} when $q\in(0,X^{1/2})$ the term $\frac{X}{q^{1/2}}(\log X)^4$ dominates, and when $q\in(X^{1/2},X)$ the term $X^{1/2}q^{1/2}$ dominates. Therefore, comparing \eqref{bound 1} and \eqref{bound 2}, we can see that \eqref{bound 1} is always better when $q\leq X$.
\end{remark}

\section{Proof of Theorem \ref{thm:exponential sum for mu squared squared} and additional results} \label{sec:sec5}
As we saw in Section \ref{sec:sec2}, two critically important results were Lemma \ref{thm:general_vinogradov} and Lemma \ref{lem:ik} with the additional refinement of being able to take $\Upsilon>0$ into account. There exist other instances where being able to accommodate the parameter $\Upsilon$ will be very helpful, especially if $\Upsilon$ is not constrained to be greater than or equal to $1$. One such instance takes place when bounding exponential sums associated to the Liouville $\lambda(n)$ function or the Jordan totient function $J_k(n)$. 
To that end, we present the following generalization to Weyl's bound, see \cite{IwKo04,Vi42}.

\begin{lemma}[Generalized Weyl bound]\label{Generalized Weyl bound}
    Let $M,N,a,q$ be integers such that $(a,q) = 1$ and $q >0$. If $f$ is a real polynomial of degree $k \geq 1$ with leading coefficient $\alpha$ such that $\lvert\alpha-a/q\rvert\leq \Upsilon q^{-2}$ for some $\Upsilon > 0$, then for any $\varepsilon>0$ we have
    \[
    \sum_{x=M+1}^{M+N} \e(f(x))\ll N^{1+\varepsilon}\bigg(\frac{\max\{1,\Upsilon\}}{q}+\frac{1}{N}+\frac{1}{N^{k-1}}+\frac{q}{N^k}\bigg)^{2^{1-k}}.
    \]
\end{lemma}
\begin{proof}
Weyl's inequality implies the case $\Upsilon =1$ and the extension to $\Upsilon >0$ follows the same principles as in the proof of Lemma~\ref{thm:general_vinogradov}.
\end{proof}

Similarly, we may show the following stronger result which is only valid for $k=2$.

\begin{lemma}[Generalized quadratic Weyl bound] \label{Generalized quadratic Weyl bound}
    Let $M,N,a,q$ be integers such that $(a,q) = 1$ and $q >0$. If $f(x)=\alpha x^2 + \beta x + \gamma$ such that $\alpha$ such that $|\alpha-a/q|\le \Upsilon q^{-2}$ for some $\Upsilon > 0$, then we have
    \[
    \sum_{x=1}^{N} \e(f(x))\ll \frac{N}{\sqrt{q}}\max\{1,\Upsilon^{1/2}\} + \sqrt{N \log q} + \sqrt{q \log q}.
    \]
\end{lemma}
\begin{proof}
The $\Upsilon = 1$ case can be found in \cite[$\mathsection$3, Theorem 1]{montgomery10}. The extension to $\Upsilon>0$ follows the same guidelines as Lemma~\ref{thm:general_vinogradov}.
\end{proof}

We can use these results to prove a bound for exponential sums twisted by $|\mu|$ but with a polynomial in $P(n)$ instead of $n$ in $\e(\cdot)$.

\begin{proof}[Proof of Theorem \textnormal{\ref{thm:exponential sum for mu squared squared}}]
    Note that 
     \begin{align} \label{original exponential sum}
        \sum_{n\leq X} \mu^2(n) \e(n^2\alpha) = \sum_{b\leq X^2} f(b)\e(b\alpha) \quad \textnormal{where} \quad     f(a)=
        \begin{cases}
            1, &\quad m = (p_1\cdots p_r)^2\\
            0, &\quad \mbox{otherwise}.
        \end{cases}
    \end{align}
    This can be seen by considering the Euler product of $f$ as
    \begin{align}
        F(s) = \sum_{n=1}^\infty \frac{f(n)}{n^s} = \prod_p \bigg(1+\frac{1}{p^{2s}}\bigg) &= \prod_p \bigg(1+\frac{1}{p^{2s}} \bigg)\bigg(1-\frac{1}{p^{2s}} \bigg)\bigg(1+\frac{1}{p^{2s}} \bigg)^{-1} \nonumber \\
        &= \prod_p \bigg(1-\frac{1}{p^{4s}} \bigg) \bigg(1-\frac{1}{p^{2s}} \bigg)^{-1} = \frac{\zeta(2s)}{\zeta(4s)},
    \end{align}
    for $\real(s) > \frac{1}{2}$.
Therefore, we may write $f = g*h$, where
\begin{align*}
    g(m)=
        \begin{cases}
            1, &\quad \mbox{if $m$ is a square}\\
            0, &\quad \mbox{otherwise}.
        \end{cases}
   \quad\quad\quad h(\ell) = 
     \begin{cases}
            \mu(\ell^{1/4}), &\quad \mbox{if $\ell$ is a fourth power}\\
            0, &\quad \mbox{otherwise}.
        \end{cases}
    \end{align*}
Then \eqref{original exponential sum} is equivalent to
\begin{align}
    \sum_{\substack{m\ell\leq X^2\\ m,\ell \geq 1}} g(m)h(\ell) \e(m\ell\alpha).
\end{align}
For generality, replace $X^2$ by $Y$. Then, our goal is to bound the exponential sum
\begin{align}
    \sum_{\substack{m\ell\leq Y\\ m \geq 1, \ell \geq 1}} g(m)h(\ell) \e(m\ell\alpha),\label{revised exponential sum}
\end{align}
where $\alpha,q$ satisfy the minor arc condition $|\alpha-\frac{a}{q}| \le q^{-2}$ and $(a,q)=1$. Since we are summing over the hyperbolic region $\Sigma:=\{m,\ell: m\ell\leq Y\}$, for arbitrary $L\geq 1$, let $M=Y/L$ and define the following sets:
\begin{align*}
    \Sigma_1 = \{m,\ell: m\leq M, \ell\leq Y/m\}, \quad
    \Sigma_2 = \{m,\ell: \ell\leq L, m\leq Y/\ell\}, \quad
     \Sigma_3 = \{m,\ell: m\leq M, \ell\leq L\}.
\end{align*}
Then, $\Sigma = \Sigma_1 +\Sigma_2-\Sigma_3$. The sum over $\Sigma_1$ is bounded by
\begin{align*}
    \sum_{m\leq M}\sum_{\ell\leq Y/m} g(m)h(\ell) \e(m\ell\alpha) &\ll \sum_{m\leq M}g(m)\sum_{u\leq (Y/m)^{1/4}}\mu(u)\e(mu^4\alpha)\\
    &\ll Y^{1/4}\sum_{v\leq \sqrt{M}} \frac{1}{v^{1/2}} \ll Y^{1/4}M^{1/4}.
\end{align*}
The sum over $\Sigma_2$ is bounded by
\begin{align*}
    \sum_{\ell\leq L}\sum_{m\leq Y/\ell} g(m)h(\ell)\e(m\ell\alpha) &\ll \sum_{\ell\leq L}h(\ell) \sum_{u\leq \sqrt{Y/\ell}} \e(u^2\ell\alpha). 
\end{align*}
Take $\Upsilon = \frac{\ell}{(\ell,q)^2}$ in Lemma \ref{Generalized quadratic Weyl bound}, the above is then bounded by 
\begin{align*}
    \sum_{\ell\leq L}h(\ell) \bigg(\frac{\sqrt{Y}}{\sqrt{q}}+\frac{Y^{1/4}}{\ell^{1/4}}\sqrt{\log q}+\sqrt{q\log q}\bigg) &\ll\sum_{v\leq L^{1/4}} \bigg(\frac{\sqrt{Y}}{\sqrt{q}}+\frac{Y^{1/4}}{v}\sqrt{\log q}+\sqrt{q\log q}\bigg)\\
    &\ll\frac{\sqrt{Y}L^{1/4}}{\sqrt{q}}+Y^{1/4}\log Y\sqrt{\log q}+L^{1/4}\sqrt{q\log q}.
\end{align*}
The third sum over $\Sigma_3$ can be bounded by 
\begin{align*}
  \sum_{\ell\leq L}\sum_{m\leq M}g(m)h(\ell)\e(m\ell\alpha) \ll \sum_{\ell\leq L} h(\ell)\sum_{u\leq M^{1/2}} \e(u^2\ell\alpha)\ll \sum_{v\leq L^{1/4}}  M^{1/2}=M^{1/2}L^{1/4}.
\end{align*}
Now setting $M=Y/L$, we see that 
\begin{align*}
    \sum_{\substack{m\ell\leq Y\\ m,\ell \geq 1}} g(m)h(\ell) \e(m\ell\alpha)&\ll \frac{Y^{1/2}}{L^{1/4}}+\frac{\sqrt{Y}L^{1/4}}{\sqrt{q}}+Y^{1/4}\log Y\sqrt{\log q}+L^{1/4}\sqrt{q\log q}.
\end{align*}
Using Lemma \ref{lem:min_max_for_minor},
the final bound is
\begin{align*}
\sum_{\substack{m\ell\leq Y\\ m,\ell \geq 1}} g(m)h(\ell) \e(m\ell\alpha)&\ll \frac{Y^{1/2}}{q^{1/4}}+Y^{1/4}\log Y(\log q)^{1/2}+Y^{1/4}q^{1/4}(\log q)^{1/4}.
\end{align*}
Substituting $Y=X^2$, we finally conclude that
\[
\sum_{n\leq X} \mu^2(n)\e(n^2\alpha) \ll \frac{X}{q^{1/4}}+X^{1/2}\log X(\log q)^{1/2}+X^{1/2}q^{1/4}(\log q)^{1/4},
\]
as we wanted to show. This ends the proof.
\end{proof}

It is worth remarking that other approaches could been undertaken in order to achieve a bound for $\tilde S_{|\mu|}$. In \cite[Theorem 1.6]{BRZ23} it was shown that 
\begin{align} \label{eq:theorem1.6bound}
    \sum_{n \le X} (\mu * g_2)(n) \e(\alpha n) \ll \bigg(\frac{X}{q^{1/2}}+X^{5/6}+X^{1/2}q^{1/2}\bigg)(\log X)^{5/2}
\end{align}
where $g_2$ is the indicator function of squarefull numbers. The Dirichlet series for $\mu * g_2$ corresponds to $\sum_{n=1}^\infty (\mu*g_2)(n)n^{-s} = \frac{\zeta(2s)\zeta(3s)}{\zeta(s)\zeta(6s)}$. One such approach would be to consider $(\mu * q_1)(n)$ or $(|\mu| * q_2)(n)$ where $\sum_{n=1}^\infty q_1(n) n^{-s} = \frac{\zeta(2s)}{\zeta(s)\zeta(4s)}$ and $\sum_{n=1}^\infty q_2(n) n^{-s} = \frac{\zeta(2s)^2}{\zeta(s)\zeta(4s)}$. This would involve using the Vaughan identity associated to $q_1$ and $q_2$. Alternatively, one could try to get a bound for $\sum_{n \le X} \tilde \mu_4(n) \e(\alpha n)$ where $\sum_{n=1}^\infty \tilde \mu_4(n) n^{-s} = (\zeta(4s))^{-1}$ for $\real(s) > \frac{1}{4}$, much like  in \eqref{eq:2328}. This would also likely involve the use of the Vaughan identity. Moreover, one would also need a bound for $\sum_{n \le X} j_2(n) \e(\alpha n)$ where $\sum_{n=1}^\infty j_2(n) n^{-s} = \zeta(2s)$ for $\real(s) > \frac{1}{2}$. Then, by the use of the hyperbola/$\Upsilon$ method, one convolves both bounds to get the desired final bound. These approaches will be the subject of future research.

By using Lemma \ref{Generalized Weyl bound} bound we may extend Theorem \ref{thm:exponential sum for mu squared squared}, although the bound will not be as precise since Weyl's inequality is weaker when dealing with a polynomial $P$ of degree $k > 2$.

\begin{theorem} \label{thm:musquaredpolynomial}
    Let $\alpha \in \R$, $a \in \Z$, $q \in \N$ such that $|\alpha-\frac{a}{q}| \le \frac{1}{q^2}$ with $(a,q)=1$. Let $X\geq 2,k>2,$ and define 
    \[
    R(P_k, \alpha, X) := \sum_{n\leq X}|\mu(n)|\e(P(n)\alpha),
    \]
    where $P_k(n) = c_kn^k+c_{k-1}n^{k-1}+\cdots+c_0$ over $\R$. Then, we have
    \begin{align*}
        R(P_k, \alpha, X)\ll \frac{X^{1/2+2^{2-k}+\varepsilon}}{q^{2^{1-k}}}+q^{2^{1-k}}X^{1/2+\varepsilon},
    \end{align*}
    for any $\varepsilon>0$. In particular, for $k=3$, we have a slightly sharper bound, which is 
    \[
    R(P_3, \alpha, X)\ll X^{1/2+\varepsilon}q^{1/6}+\frac{X^{1+\varepsilon}}{q^{1/4}},
    \]
    for any $\varepsilon>0$. 
\end{theorem}
\begin{proof}
The proof is similar to that of Theorem \ref{thm:exponential sum for mu squared squared} except that we use Lemma \ref{Generalized Weyl bound} instead of Lemma \ref{Generalized quadratic Weyl bound} to bound the sum over $\Sigma_2$. We then see that the sum over $\Sigma_2$ is bounded by 
\[
L^{2^{1-k}-1/4}q^{-2^{1-k}}Y^{1/2+\varepsilon}+L^{2^{-k}-1/4}Y^{-2^{-k}+1/2+\varepsilon}+L^{k2^{-k}-1/4}q^{2^{1-k}}Y^{-k2^{-k}+1/2+\varepsilon},
\]
for any $\varepsilon>0$.
Combining all bounds, we obtain
\begin{align}
    \sum_{\substack{m\ell\leq Y\\ m,\ell \geq 1}} g(m)h(\ell) \e(m\ell\alpha)&\ll \frac{Y^{1/2}}{L^{1/4}}+L^{2^{1-k}-1/4}q^{-2^{1-k}}Y^{1/2+\varepsilon}+L^{2^{-k}-1/4}Y^{-2^{-k}+1/2+\varepsilon} \nonumber \\
    &\quad +L^{k2^{-k}-1/4}q^{2^{1-k}}Y^{-k2^{-k}+1/2+\varepsilon}.\label{bound in thm 5.1}
\end{align}
Using Lemma \ref{lem:min_max_for_minor} when $k=3$ implies that \eqref{bound in thm 5.1} is bounded by $Y^{1/4+\varepsilon}q^{1/6}+Y^{1/2+\varepsilon}q^{-1/4}$. When $k\geq 4$, \eqref{bound in thm 5.1} is bounded by \[
Y^{1/4}+q^{-2^{1-k}}Y^{1/4+2^{1-k}+\varepsilon}+Y^{1/4+\varepsilon}+q^{2^{1-k}}Y^{1/4+\varepsilon}\ll q^{-2^{1-k}}Y^{1/4+2^{1-k}+\varepsilon}+q^{2^{1-k}}Y^{1/4+\varepsilon}, 
\]
for any $\varepsilon>0$. Substituting $Y=X^2$, we get our desired result.
\end{proof}

As illustrated in Section \ref{sec:sec4}, it is likely that the bounds we have presented, while useful for most applications where demanding to very demanding savings are required, could further be improved by decreasing the exponents associated to $X, \log X$ and $q$ using specifically tailored arguments for the problem at hand. A useful tool in this direction will likely involve the Heath-Brown formula for decomposing $\mu$ and $\Lambda$ as well as the Selberg identity for $\Lambda_2$, see \cite{heathbrownformula} and \cite[$\mathsection$13]{IwKo04}, both of which are advanced combinatorial decompositions of arithmetic functions. Indeed the Heath-Brown decomposition of $\Lambda$ and $\mu$ played a key role in the enlarging the size of the mollifier of the second moment of the Riemann zeta-function in \cite{prattRobles, przz}. In turn, these enlargements increased the proportion of non-trivial zeros of $\zeta$ on the critical line. Here we present an analogue for $\Ip(n)$ that will become useful in future research.

\begin{lemma}
Let $k\in\N, x\geq 1$ and $V\geq x^{1/k}$. For $n\leq x$, we have
\[
 \Ip(n) = \sum_{j=0}^k (-1)^{j-1} {k\choose j}{\mu_{\leq V}}^{*j}*\mathbf{1}^{*(j-1)}*\omega.
\]
\end{lemma}
\begin{proof}
    Write $\mu = \mu_{\leq V}+\mu_{>V}$, then clearly the convolution
    $
    \mu_{>V}^{*k}*\mathbf{1}^{*k-1}*\omega
    $
    vanishes on $\N_{\leq x}$. Expanding out $\mu_{>V} = \mu - \mu_{\leq V}$ and using the binomial formula, we get
    \begin{align}
       0 = \sum_{j=0}^k (-1)^j {k\choose j}\mu^{*(k-j)}*{\mu_{\leq V}}^{*j}*\mathbf{1}^{*(k-1)}*\omega.\label{binomial formula result}  
    \end{align}
    Using the identities $\mu*\omega = \Ip$ and $\mu*1 = \delta$, where $\delta$ is the multiplicative identity of Dirichlet convolution, the $j=0$ term of \eqref{binomial formula result} becomes 
    $
    \mu^{k}*\mathbf{1}^{*(k-1)}*\omega = \Ip
    $.
    For all other terms of \eqref{binomial formula result}, we use again the identity $\mu*1 = \delta$ and write 
    \[
    \mu^{*(k-j)}*{\mu_{\leq V}}^{*j}*\mathbf{1}^{*(k-1)}*\omega = {\mu_{\leq V}}^{*j}*\mathbf{1}^{*(j-1)}*\omega.
    \]
  Combining all, \eqref{binomial formula result} can be rewritten as in the statement of the lemma.
\end{proof}

\section{Proof of Theorem \ref{thm:partitionTheorem}} \label{sec:partitionproof}
Results regarding weighted or signed partitions can be found in recent literature \cite{BRZ23,colorsDivisor,taylorMoebius,drzz2}. The strategy consists in employing the Hardy-Littlewood circle method to extract the asymptotics of $\mathfrak{p}$. Generally, although not always, this process results in examining a contour integral over three distinct arcs. For brevity we shall provide the hardest ingredient for each arc and the reader is referred to the above mentioned literature for the remaining details.

Equation \eqref{eq:generatingPartitions} can be rewritten as
\[
\Psi(z) = \exp(\Phi(z)) \quad \textnormal{where} \quad \Phi(z) = \sum_{j=1}^\infty \sum_{n=1}^\infty \frac{|\mu(n)|}{j} z^{jn}.
\]
We may now use Cauchy's integral formula to extract $\mathfrak{p}_{|\mu|}$ as
\begin{align} \label{eq:CauchyIntegralCircle}
    \mathfrak{p}_{|\mu|}(n)= \rho^{-n} \int_{0}^{1} \Psi(\rho \e(\alpha)) \e(-n \alpha)d\alpha = \rho^{-n} \int_{0}^{1} \exp[\Phi(\rho \e(\alpha)) - 2\pi i n \alpha] d\alpha,
\end{align}
for any $0<\rho<1$. Let $x \in \R$ be large. We choose the radius $\rho$ to be $\rho = \rho(x)$ so that $x = \rho \Phi'(\rho)$. 
Furthermore, we we can see that
\begin{align}
    \rho\Phi'(\rho) = \sum_{j=1}^\infty \sum_{n=1}^\infty n|\mu(n)| \rho^{jn}.
\end{align}
Therefore, $\rho\Phi'(\rho)$ is a strictly monotonically increasing function on the interval $[0,1)$ and $\lim_{\rho\to 1} \Phi'(\rho) =\infty$.
Thus the relationship between $x$ and $\rho$ is well-defined and injective.

\subsection{Set up of the arcs} \label{subsec:arcsDefinition}
We now define the major arcs $\mathfrak{M}$ and minor arcs $\mathfrak{m}$ as follows. For real $A>0$ we set
\begin{align}
\delta_q = q^{-1}X^{-1}(\log X)^A \quad \textnormal{and} \quad Q = (\log X)^A.
\end{align}
Moreover, for $1 \le a \le q \le Q$ with $(a,q)=1$ we define $\mathfrak{M}(q,a)$ to be the open interval
\begin{align}
\mathfrak{M}(q,a) := \bigg(\frac{a}{q}-\delta_q, \frac{a}{q}+\delta_q\bigg).
\end{align}
The major arcs $\mathfrak{M}$ and minor arcs $\mathfrak{m}$ are then defined by
\begin{align}
\mathfrak{M} = \bigcup_{1 \le q \le Q}\bigcup_{\substack{1 \leq a \leq q\\ (a,q)=1}} \mathfrak{M}(q,a)  \quad \textnormal{and} \quad \mathfrak{m} = [-\tfrac{1}{2},\tfrac{1}{2}) \backslash \mathfrak{M}.
\end{align}
We use the canonical splitting of \eqref{eq:CauchyIntegralCircle} as
\begin{align}
    \mathfrak{p}_{|\mu|}(n) = \rho^{-n}\bigg(\int_{\mathfrak{M}(1,0)}+\int_{\mathfrak{M}\backslash\mathfrak{M}(1,0)}+\int_{\mathfrak{m}}\bigg) \exp[\Phi(\rho \e(\alpha)) - 2\pi i n \alpha] d\alpha.
\end{align}
For the first integral, we shall use contour integration and the saddle point method, for the second one we will use an adapted Siegel-Walfisz lemma associated to $|\mu|$, and for the third term we will exploit the bounds on exponential sums established earlier.

\subsection{The explicit formula}
In order to prove the explicit formula of this section, we will require a contour integration involving the term $1/\zeta(2s)$. To prove our formula unconditionally, we shall use the following result due to the works of Ramachandra and Sankaranarayanan \cite[Theorem 2]{RamachandraSankaranarayanan} (see also \cite[Lemma 1]{Inoue} and \cite[Theorem 1.5]{kuhnrobles}). This result provides unconditional upper bounds for $1/\zeta(s)$ inside the critical strip along a specific vertical line segment.

\begin{lemma} \label{lem:RamachandraSankaranarayanan}
    Let $T>0$ be a sufficiently large positive number and $H=T^{1/3}$. Then one has
    \begin{align*}
        \min_{T \leqslant t \leqslant T+H} \max_{\frac{1}{2} \leqslant \sigma \leqslant 2} |\zeta(\sigma+it)|^{-1} \leqslant \exp( C (\log \log T)^2)
    \end{align*}
    with an absolute constant $C>0$. In particular, there exists a real $T_* \in [T,T+T^{1/3}]$ such that
    \begin{align*}
        \frac{1}{\zeta(\sigma+iT_*)} \leqslant T_*^\varepsilon \quad \textnormal{where} \quad \frac{1}{2} \leqslant \sigma \leqslant 2,
    \end{align*}
    for any $\varepsilon>0$.
\end{lemma}
We shall define a parameter $X = (\log \frac{1}{\rho})^{-1}$ so that $\rho = e^{-1/X}$ as well as $\Delta := (1+4\pi^2 X^2 \theta^2)^{-1/2}$.

\begin{lemma} \label{lem:lemmaMainTerms}
    Suppose that $\theta \in \R$ and $X \ge 1$. If $X \Delta ^3 \ge 1$, then there exists a sequence $T_{\nu}$ with $\nu \le T_\nu \le 2\nu$ such that
    \begin{align} \label{eq:maintermslemmazeros}
        \Phi_{|\mu|}(\rho \e (\theta)) = \frac{X}{1-2\pi i X \theta} + \log \frac{X}{2\pi(1-2\pi i X \theta)} + \lim_{\nu \to \infty}\sum_{|\imag(\varpi)| < T_\nu} f(X,\theta,\varpi) + \sum_{n=1}^\infty g(X,\theta,n),
    \end{align}
    where the functions $f$ and $g$ are given by
    \begin{align}
        f(X,\theta,\varpi) &= \frac{1}{2 \zeta'(\varpi)}\bigg(\frac{X}{1-2\pi i X \theta}\bigg)^{\varpi/2} \Gamma\bigg(\frac{\varpi}{2}\bigg) \zeta\bigg(1+\frac{\varpi}{2}\bigg) \zeta\bigg(\frac{\varpi}{2}\bigg), \nonumber \\
        g(X,\theta,n) &= \bigg(\mathfrak{c}_1(n) \log \frac{X}{1-2\pi i X \theta} + \mathfrak{c}_2(n) \bigg)\bigg(\frac{X}{1-2\pi i X \theta}\bigg)^{-n}. \nonumber
    \end{align}
    The coefficients $\mathfrak{c}_1$ and $\mathfrak{c}_2$ are computable constants defined in \eqref{eq:defCtrivial}. The sum is taken over the non-trivial zeros $\varpi$ of $\zeta(s)$ under the assumption, for notational ease, of simplicity.
\end{lemma}
\begin{remark}
    Unlike assuming the Riemann hypothesis, the assumption of the simplicity of the zeros does not represent a technical difficulty and it can be relaxed at the expense of cluttering the expression for $f(X,\theta,\varpi)$ to the point of becoming unreadable, see \cite[$\mathsection$14.27]{Titchmarsh1986} `obvious modifications are required if the zeros are not simple'. 
\end{remark}
\begin{proof}[Proof of Lemma \textnormal{\ref{lem:lemmaMainTerms}}]
    Recall that
    \begin{align}
        \Phi_{|\mu|}(\rho \e (\theta)) = \sum_{j=1}^\infty \sum_{n=1}^\infty \frac{|\mu(n)|}{j} \exp\bigg(-jn \bigg(\frac{1}{X}-2\pi i \theta\bigg)\bigg),
    \end{align}
    where $\rho = e^{-1/X}$. Applying the inverse Mellin transform of the exponential function we obtain
    \begin{align}
        \Phi_{|\mu|}(\rho \e (\theta)) = \sum_{j=1}^\infty \sum_{n=1}^\infty \frac{|\mu(n)|}{j} \frac{1}{2\pi i} \int_{(c)} \Gamma(s) j^{-s}n^{-s} \bigg(\frac{X}{1-2\pi i X \theta}\bigg)^s ds,
    \end{align}
    for $c>0$. The interchange of summation and integration is legal for any $c>1$ and it leads to
    \begin{align} \label{eq:PathIntegralc}
    \Phi_{|\mu|}(\rho \e(\theta)) &= \frac{1}{2 \pi i} \int_{(c)} \Gamma(s) \bigg(\sum_{j=1}^\infty \frac{1}{j^{s+1}}\bigg) \bigg(\sum_{n=1}^\infty \frac{|\mu(n)|}{n^s} \bigg) \bigg(\frac{X}{1-2\pi i X \theta}\bigg)^s ds \nonumber \\
    &= \frac{1}{2 \pi i} \int_{(c)} \Gamma(s) \zeta(s+1) \frac{\zeta(s)}{\zeta(2s)} \bigg(\frac{X}{1-2\pi i X \theta}\bigg)^s ds. 
    \end{align}
    We temporarily set $y = \frac{X}{1-2\pi i X \theta}$ and proceed with a singularity analysis. The path of integration in \eqref{eq:PathIntegralc} is the vertical line $\real(s)=c>1$. The technique is to now push the part $|t| \le T$ of this path to the left using a rectangular contour whose vertices are $c \pm iT$ and $-u \pm iT$ where $u=N+\frac{1}{2}$, for some large $N \in \N$ in order to avoid the trivial zeros. The new contour of integration is therefore $\gamma = \cup_{i=1}^5 \gamma_i$,where $\gamma_1=(c-i\infty, c-iT]$, $\gamma_2=[c-iT,-u-iT]$, $\gamma_3=[-u-iT,-u+iT]$, $\gamma_4=[-u+iT,c+iT]$ and $\gamma_5 =[c+iT,c+i\infty]$. Using the Cauchy residue theorem leads to
    \begin{align}
        \Phi_{|\mu|}(\rho \e (\theta)) = R + \sum_{k=1}^5 \frac{1}{2\pi i}\int_{\gamma_k} \Gamma(s) \zeta(s+1)\frac{\zeta(s)}{\zeta(2s)}y^s ds = R +  \sum_{k=1}^5 I_k,
    \end{align}
    where $R$ is equal to the contribution of the residues at the singularities of the integrand in the region enclosed by the two paths of integration. We must now look at each integral separately and bound it and find the value of $R$.
    
    We start with $R$. At $s=1$ we have a simple pole coming from $\zeta(s)$ for which we have
    \begin{align}
        \mathop{\operatorname{res}}\limits_{s=1} \Gamma(s) \zeta(s+1) \frac{\zeta(s)}{\zeta(2s)} y^s = y,
    \end{align}
    which will serve as leading term. At the non-trivial zeros $s = \varpi/2$ we have poles which we assume to be simple for which
    \begin{align}
        \mathop{\operatorname{res}}\limits_{s=\varpi/2} \Gamma(s) \zeta(s+1) \frac{\zeta(s)}{\zeta(2s)} y^s = \frac{y^{\varpi/2}}{2 \zeta'(\varpi)} \Gamma\bigg(\frac{\varpi}{2}\bigg) \zeta\bigg(1+\frac{\varpi}{2}\bigg) \zeta\bigg(\frac{\varpi}{2}\bigg) .
    \end{align}
    There is also a double pole from $\Gamma(s)\zeta(s+1)$ at $s=0$ which yields
    \begin{align}
        \mathop{\operatorname{res}}\limits_{s=0} \Gamma(s) \zeta(s+1) \frac{\zeta(s)}{\zeta(2s)} y^s = \log \frac{y}{2\pi}.
    \end{align}
    For the non-trivial zeros at $s=-n$ for $n=1,2,3,\cdots$ we also have double poles for which
     \begin{align} \label{eq:defCtrivial}
   \mathop{\operatorname{res}}\limits_{s=-n} \Gamma(s) \zeta(s+1) \frac{\zeta(s)}{\zeta(2s)} y^s &= \frac{(-1)^n}{ 
 2 n! }  \frac{\zeta (1-n) \zeta (-n)}{\zeta '(-2 n)}\nonumber \\
    &\quad \times \left(\log y 
   -\frac{\zeta ''}{\zeta'}(-2 n)
   + \frac{\zeta '}{\zeta}(1-n)
   + \frac{\zeta '}{\zeta}(-n)
   + \psi(n+1) \right)y^{-n} \nonumber \\
   &=: (\mathfrak{c}_1(n) \log y + \mathfrak{c}_2(n))y^{-n},
    \end{align}
    where, we recall, $\psi(x)=\frac{\Gamma'}{\Gamma}(x)$ denotes the digamma function. We now have $R$.

    Except for one part of the horizontal integral over $\gamma_4$, the rest of the analysis is straightforward. We shall make sure of $|y|^s \le (X \Delta)^{-u}e^{-|t|(\Delta-\pi/2)}$ from \cite{gafniPowers}.
    For instance, to bound the vertical segments $\gamma_1$ and $\gamma_5$, we employ $\Gamma(s) \ll e^{-\pi T/2}$, $\zeta(s+1) \ll 1$, $\zeta(2s)^{-1} \ll 1$ and $|y|^s \ll (X\Delta)^\sigma e^{-T(\Delta-\pi/2)}$. These imply that $I_1 \ll e^{-T\Delta/2}$, with a similar bound holding for $I_5$.
    For the bound the vertical integral over the segment $\gamma_3$, we use Stirling's formula $\Gamma(s) \zeta(s+1) \zeta(s) \ll (2\pi)^{2\sigma}|s|^{-\sigma-1/2}e^{-\pi|t|/2}$ as well as the functional equation of $\zeta(2s)$. These lead to $I_3 \ll X^{-u}\Delta^{-2u}(u^{-2u-3/2}\Delta^{-1/2})$.
    The difficulty resides in the horizontal integrals over the paths $\gamma_2$ and $\gamma_4$. Let us consider $I_4$ since the treatment for $I_2$ is nearly identical. One first splits the path into
    \[
    I_4 = \bigg(\int_{-u+iT}^{-3/2+iT}+\int_{-3/2+iT}^{-1+iT}+\int_{-1+iT}^{c+iT}\bigg)\Gamma(s) \zeta(s+1)\frac{\zeta(s)}{\zeta(2s)}y^s ds = I_{4,1}+I_{4,2}+I_{4,3},
    \]
    say. The integral $I_{4,1}$ is actually not difficult and it follows by analytic techniques, i.e. using the functional equation of $\zeta(s)$ as well as $(\Gamma(s-1))^{-1} \ll e^{1-\sigma-(\frac{1}{2}-\sigma)\log(1-\sigma)+\frac{1}{2}\pi T}$ yields $I_{4,1} \ll e^{-T\Delta/2}$. The bound for $I_{4,2}$ is also $I_{4,2} \ll e^{-T\Delta/2}$ as can be seen by the functional equation. However, to deal with $I_{4,3}$ we employ Lemma \ref{lem:RamachandraSankaranarayanan}: for $-1 \le \real(s) = \sigma \le 2$ and any $\varepsilon>0$, for sufficiently large $V>0$, there exists some $V \le T_* \le 2V$ such that
    \begin{align} \label{eq:boundRSepsilon}
    \frac{1}{\zeta(2(\sigma+ iT_*))} \ll T_*^\varepsilon.
    \end{align}
    The above bound along with $\Gamma(s) \ll T^{\sigma-1/2}e^{-\pi T/2}$ and $\zeta(s+1) \ll T^\eta$ for some fixed constant $\eta>0$ yields $I_{4,3} \ll e^{-T \Delta/2}$, which was the last piece.

    The final step is to combine all these estimates and let $N \to \infty$ so that
    \begin{align*}
        \Phi_{|\mu|}(\rho \e (\theta)) &= \frac{X}{1-2\pi i X \theta} + \log \frac{X}{2\pi(1-2\pi i X \theta)} \nonumber \\
        &\quad + \sum_{|\imag(\varpi)| < T_\nu} f(X,\theta,\varpi) + \sum_{n=1}^\infty g(X,\theta,n) + O(e^{-T_\nu \Delta /2}),
    \end{align*}
    where $\nu \le T_\nu \le 2\nu$ with $\nu \in \N$. Letting $\nu \to \infty$ and choosing our sequence of $T_\nu$ appropriately such that \eqref{eq:boundRSepsilon} is satisfied we are led to
    \begin{align*}
        \Phi_{|\mu|}(\rho \e (\theta)) = \frac{X}{1-2\pi i X \theta} + \log \frac{X}{2\pi(1-2\pi i X \theta)} + \lim_{\nu \to \infty}\sum_{|\imag(\varpi)| < T_\nu} f(X,\theta,\varpi) + \sum_{n=1}^\infty g(X,\theta,n),
    \end{align*}
    which was the last step we needed to show.
\end{proof}

Lemma \ref{lem:lemmaMainTerms} can be generalized a little more and we may show without much effort that
\begin{align}
    \bigg(\rho \frac{d}{d\rho}\bigg)^m \Phi_{|\mu|}(\rho \e (\theta)) 
    &= 
    \bigg(\frac{X}{1-2\pi i X \theta}\bigg)^{1+m} +  W(X,\theta,m) \nonumber \\
    &\quad + \lim_{\nu \to \infty}\sum_{|\imag(\varpi)| < T_\nu} f_m(X,\theta,\varpi) + \sum_{n=1}^\infty g_m(X,\theta,n),
    \label{eq:Phi^(m)}
\end{align}
for $m \ge 0$. Here $W$ is given by the cases
\begin{align}
    W(X,\theta,m) = 
    \begin{cases}
    \log \frac{X}{2\pi(1-2\pi i X \theta)}, &\quad \mbox{if $m=0$}, \nonumber \\
    \Gamma(m)(\frac{X}{1-2\pi i X \theta})^{m} , &\quad \mbox{if $m \ge 1$}, \nonumber
    \end{cases}
\end{align}
and $f_m$ is given by
\begin{align}
    f_m(X,\theta,\varpi) &= \frac{1}{2 \zeta'(\varpi)}\bigg(\frac{X}{1-2\pi i X \theta}\bigg)^{m+\varpi/2} \Gamma\bigg(\frac{2m+\varpi}{2}\bigg) \zeta\bigg(1+\frac{\varpi}{2}\bigg) \zeta\bigg(\frac{\varpi}{2}\bigg). \nonumber
\end{align}
If $m=0$ then $g_0(X,\theta,n) = g(X, \theta, n)$ and otherwise
\begin{align}
    g_m(X,\theta,n) = \frac{(-1)^n y^{-n} \zeta(-m-n)\zeta(1-m-n)}{n! \zeta(-2(m+n))} \quad \textnormal{for} \quad m \ge 1.
\end{align}
We can now use \eqref{eq:Phi^(m)} to determine the asymptotic behaviour of $\rho$.
Using that $\rho = e^{-1/X}$ and inserting $m=1$ gives $\rho\Phi'(\rho) = X^2 + O(X^{3/2+\varepsilon})$ with $\varepsilon>0$ arbitrary.
This expansion does not assume the Riemann hypothesis or the simplicity of zeros.
By the definition of $\rho(x)$, we have to solve the equation $x=\rho\Phi'(\rho)$.
Using the above asymptotics, we get
\begin{align}
X = x^{\frac{1}{2}} + O(x^{\frac{1}{4}+\varepsilon})
\label{eq:saddle_pint_solution}
\end{align}
with $\varepsilon>0$ arbitrary.
On the other hand, additional assumptions on the zeros of $\zeta$ give a lower error term in \eqref{eq:saddle_pint_solution}.
More precisely, if $\real(\varpi)\leq h$ for some $\frac{1}{2}\leq h \leq 1$ and for all zeros $\varpi$ of $\zeta$ then $\rho\Phi'(\rho) = X^2 + O(X^{1+\frac{h}{2}+\varepsilon})$ and hence
\begin{align}
X = x^{\frac{1}{2}} + O(x^{\frac{h}{4}+\varepsilon}).
\label{eq:saddle_pint_solution2}
\end{align}
In particular if the Riemann hypothesis is true, then $h=\frac{1}{2}$ and thus $X=x^{\frac{1}{2}}+O(x^{\frac{1}{8}+\varepsilon})$.

\subsection{Numerical tests on the zeros of $\zeta$}
Let us consider the truncated arithmetic double sum
\begin{align} \label{eq:arithemeticDoubleTruncated}
    \Phi_1(X, \theta, J, N) = \sum_{j=1}^J\sum_{n=1}^N \frac{|\mu(n)|}{j} \exp \bigg(-jn \bigg( \frac{1}{X}-2\pi i \theta \bigg)\bigg)
\end{align}
as well as its `analytic' counterpart
\begin{align}  \label{eq:LongExplicit}
    \Phi_2(X, \theta, T, N) = \Phi_{2,0}(X, \Theta) + \sum_{|\imag(\varpi)| < T} f(X,\theta,\varpi) + \sum_{n=1}^N g(X,\theta,n),
\end{align} 
where the leading term $\Phi_{2,0}(X, \theta)$ is defined by
\begin{align} \label{eq:ShortExplicit}
    \Phi_{2,0}(X, \theta) = \frac{X}{1-2\pi i X \theta} + \log \frac{X}{2\pi(1-2\pi i X \theta)}.
\end{align}
Here we have set $\theta=\theta(X)$ to be 
\begin{align}
    \theta(X) = \frac{1}{2\pi} \sqrt{\frac{1}{X^{4/3}}-\frac{1}{X^2}},
\end{align}
so that the condition $X \Delta^3 =1$ is satisfied. We can now numerically illustrate the effects of the zeros of the Riemann zeta-function on $\Phi$ in the plots below. In Figure \ref{fig:plotsShortLongAnalytical}, we can appreciate the difference (in the real and imaginary parts) between the `full' explicit formula \eqref{eq:LongExplicit} and just the leading term \eqref{eq:ShortExplicit} without the presence of the zeros of the zeta function. The contributions from the zeros of the zeta function are subtle but nevertheless present. Figures \ref{fig:plotsExplicitReal} and \ref{fig:plotsExplicitImaginary} show the difference (in real and imaginary parts, respectively) between the arithmetical component in \eqref{eq:arithemeticDoubleTruncated} as well as the analytic one in \eqref{eq:LongExplicit}.

\begin{figure}[h]
    \centering
    \includegraphics[scale=0.43]{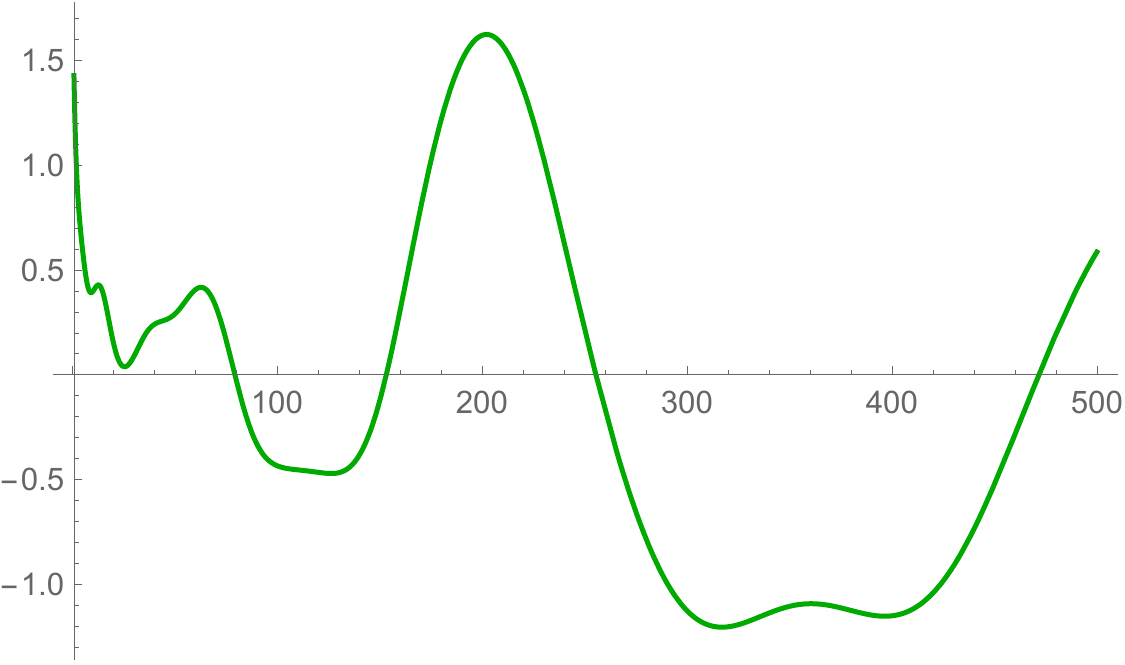}
    \includegraphics[scale=0.43]{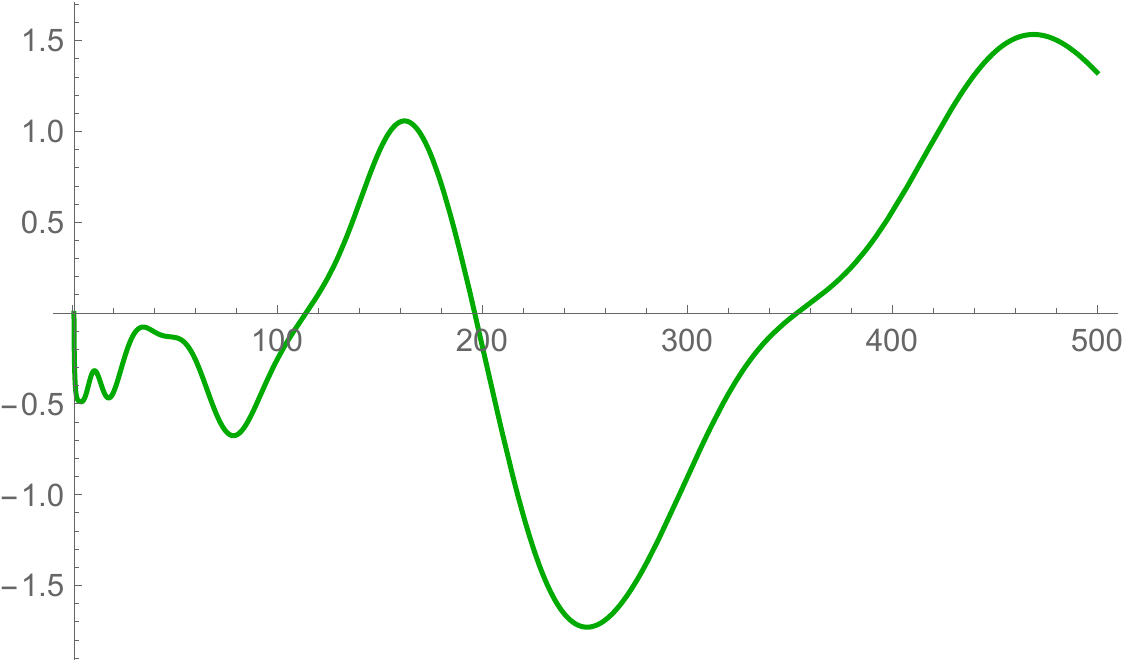}
    \caption{Plot of $\real\Phi_2(X,\theta,25,1)-\real\Phi_{2,0}(X,\theta)$ on the left-hand side and $\imag\Phi_2(X,\theta,25,1)-\imag\Phi_{2,0}(X,\theta)$ on the right-hand side for $1 \le x \le 500$.}
    \label{fig:plotsShortLongAnalytical}
\end{figure}

\begin{figure}[h]
    \centering
    \includegraphics[scale=0.43]{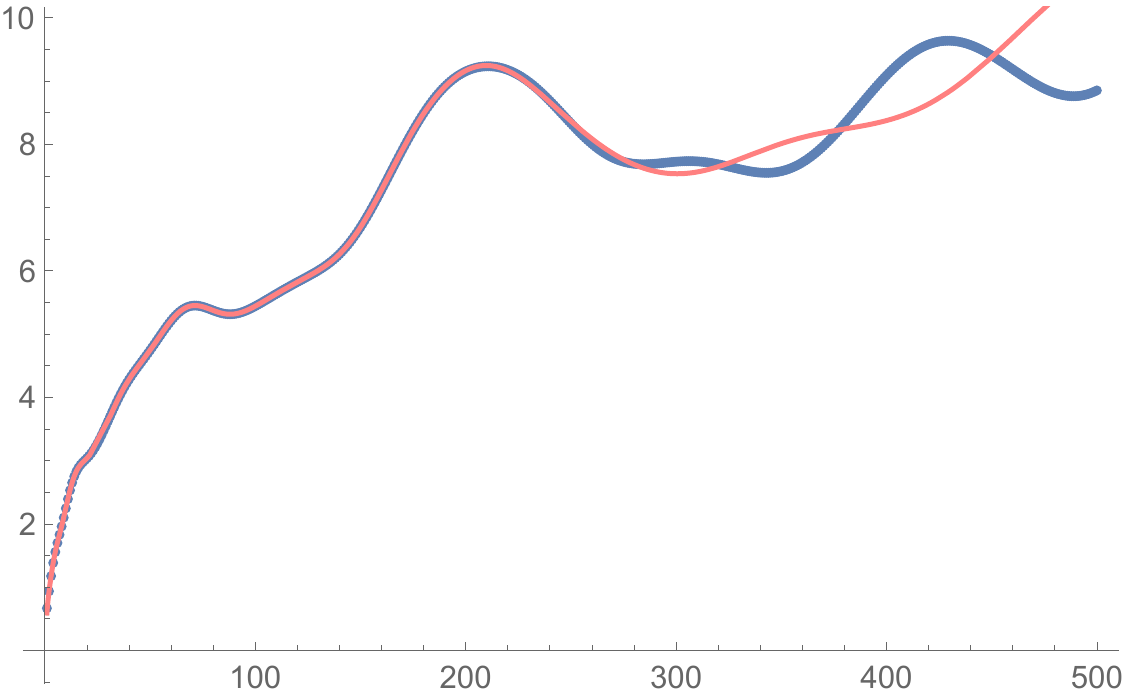}
    \includegraphics[scale=0.43]{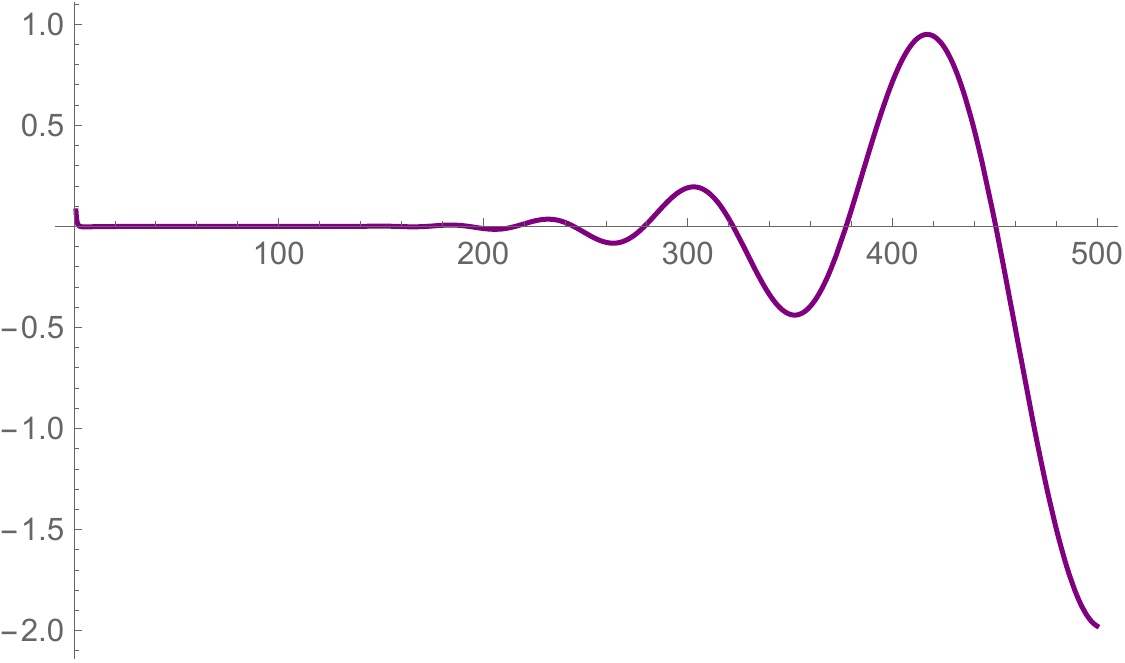}
    \caption{\underline{Left}: plot of $\real\Phi_1(X,\theta,1500,1500)$ in blue and $\real\Phi_2(X,\theta,25,1)$ in pink. \underline{Right}: plot of $\real\Phi_1(X,\theta,1500,1500)-\real\Phi_2(X,\theta,25,1)$ for $1 \le X \le 500$.}
    \label{fig:plotsExplicitReal}
\end{figure}

\begin{figure}[h]
    \centering
    \includegraphics[scale=0.43]{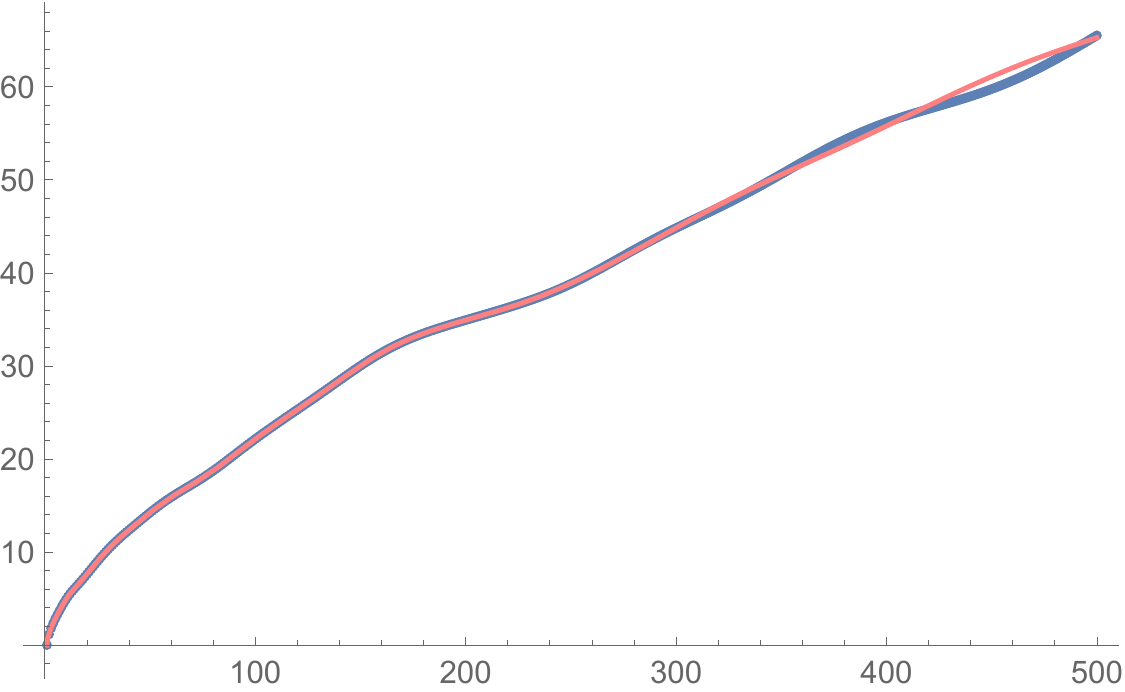}
    \includegraphics[scale=0.43]{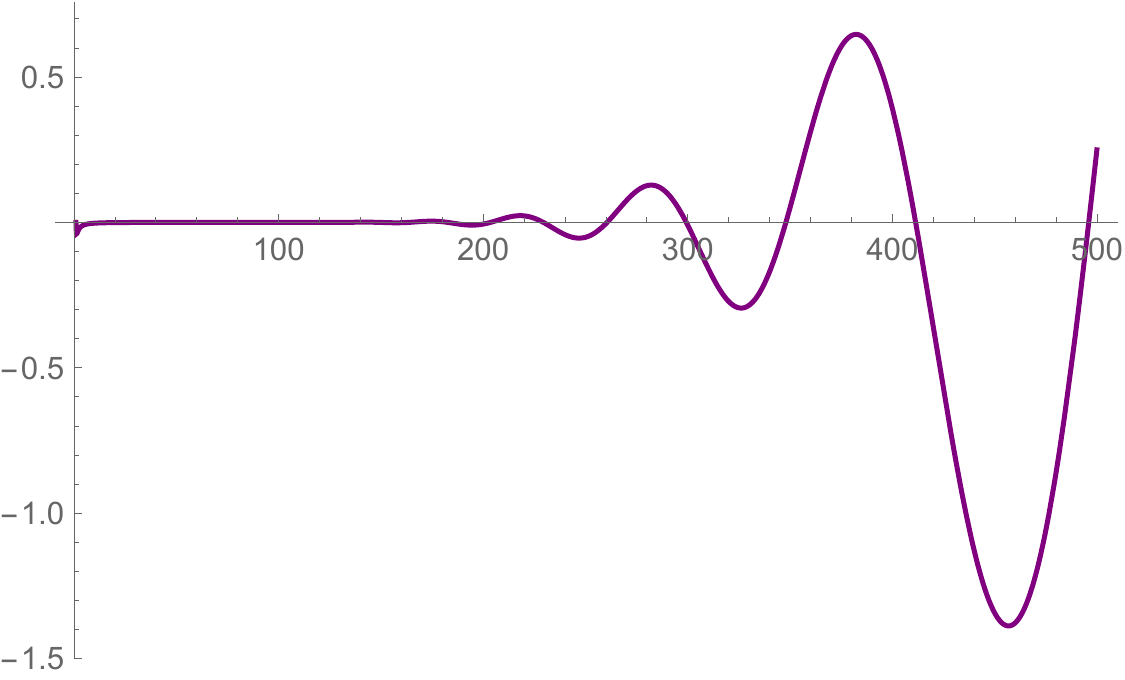}
    \caption{\underline{Left}: plot of $\imag\Phi_1(X,\theta,1500,1500)$ in blue and $\imag\Phi_2(X,\theta,25,1)$ in pink. \underline{Right}: plot of $\imag\Phi_1(X,\theta,1500,1500)-\imag\Phi_2(X,\theta,25,1)$ for $1 \le X \le 500$.}
    \label{fig:plotsExplicitImaginary}
\end{figure}

\subsection{The non-principal major arcs} \label{sec:sec|mu|nonpricipal}
Our next result involves $|\mu|$ in arithmetic progressions and it as follows.
\begin{lemma}
\label{thm:Siegel_Wal_|mu|}
As $X\to \infty$ assume that $\ell,q\in\N$ with $q \leq (\log X)^A$ for
some $A > 0$ and $(q, \ell) = 1$.
Then there exists a constant $C_A>0$ only depending $A$ such that
\begin{align}
	\sum_{\substack{n\leq X\\n \equiv \ell \modu q} } |\mu(n)|
	= 
	\frac{X}{\zeta(2)\varphi(q)\prod_{p|q}(1+p^{-1})} +O(X \exp(-C_A\sqrt{\log X}))
\end{align}
as $X \to\infty$. Moreover, the implicit constant in $O(\cdot)$ can be chosen independently of $q$ and $\ell$.
\end{lemma}
\begin{proof}
We denote by $\chi$ a Dirichlet character modulo $q$.
Using properties of characters, we can write
\begin{align}
	\sum_{\substack{n\leq X\\n \equiv a \modu q} } |\mu(n)|
	=
	\frac{1}{\varphi(q)} \sum_{\chi} \bar{\chi}(a) 	\sum_{n\leq X} \chi(n) |\mu(n)|,
\end{align}
where the outer sum  is over all Dirichlet characters $\chi$ modulo $q$.
Thus it is sufficient to show that for the principal character $\chi_0$ we have
\begin{align}
	\sum_{\substack{n\leq X} } \chi_0(n)|\mu(n)|
	= 
	\frac{X}{\zeta(2)\prod_{p\,|\,q}(1+p^{-1})} +O(X \exp(-C_A\sqrt{\log X})),
\end{align}
and for all characters $\chi\neq \chi_0$
\begin{align}
	\sum_{\substack{n\leq X} } \chi(n)|\mu(n)|
	\ll 
	X \exp(-C_A\sqrt{\log X}),
\end{align}
It is straight forward to see that for $\real(s)>1$ one has
\begin{align}
\sum_{n=1}^\infty \frac{\chi(n)|\mu(n)|}{n^s}
=
\frac{L(s,\chi)}{L(2s,\chi^2)},
\end{align}
where $L(s,\chi)$ is the Dirichlet $L$-function for the character $\chi$.
Thus Perron's formula implies 
\begin{align}
	\sideset{}{'}\sum_{\substack{n\leq X} } \chi(n)|\mu(n)|
	=
	\frac{1}{2\pi i}
	\int_{(c)}
	\frac{L(s,\chi)}{L(2s,\chi^2)} \frac{X^s}{s}\,ds
\end{align}
with $c>1$ and where the prime indicates that the last term of the sum must be multiplied by $\frac{1}{2}$ when $X$ is an integer. 
Observe that $L(2s,\chi^2)$ has no zeros for $\real(s)> \frac{1}{2}$ and the integrand has thus at most a pole at $s=1$ for $\real(s)\ge \frac{1}{2}$. 
Using the same argument as in the ordinary Siegel-Walfisz theorem to evaluate this integral, see for instance \cite[Chapter~11.3]{MV77}, we obtain
\begin{align}
	\frac{1}{2\pi i}
	\int_{(c)}
	\frac{L(s,\chi)}{L(2s,\chi^2)} \frac{X^s}{s}\,ds
	=
	\frac{1}{2\pi i}\oint_{\gamma(1)} \frac{L(s,\chi_0)}{L(2s,\chi_0^2)} \frac{X^s}{s}ds
	+
	O(X \exp(-C_A\sqrt{\log X})),
\end{align}
where $\gamma(1)$ is a small circle centered at $s=1$. If $\chi$ is a non-principal character then $L(s,\chi)$ is entire and thus
the residue at $s=1$ is $0$.
On the other hand, $L(s,\chi_0)$ has for the principal character a pole at $s=1$ with residue $\frac{\varphi(q)}{q}$ and 
$L(s,\chi_0)=\zeta(s)\prod_{p\,|\,q}(1-p^{-s})$.
Thus the residue at $s=1$ is 
\begin{align}
\mathop{\operatorname{res}}\limits_{s=1} \frac{L(s,\chi_0)}{L(2s,\chi_0^2)} \frac{X^s}{s}
=
\frac{\varphi(q)}{q}
\frac{X}{\zeta(2)\prod_{p\,|\,q}(1-p^{-2})}
=
\frac{X}{\zeta(2)\prod_{p\,|\,q}(1+p^{-1})}.
\end{align}
This completes the proof.
\end{proof}

In \cite{drzz2} and \cite{gafniPrimePowers}, it is shown that the above type of result is the main tool needed to bound the integral over the non-principal major arcs $\mathfrak{M}\backslash \mathfrak{M}(1,0)$. Once Lemma \ref{thm:Siegel_Wal_|mu|} is established the rest of the argument is very direct.
The first step is to establish the following lemma.
\begin{lemma} 
\label{lem:main_sum_exp_gamma}
Let $\gamma = \gamma_1+i\gamma_2$ with $\gamma_1 >0$ and $\gamma_2 \ll \gamma_1 (\log(1/\gamma_1))^A$ for some $A>0$ as $\gamma\to 0$. 
Moreover let $q,\ell\in\N$ with $q \ll (\log(1/\gamma_1))^A$ and $(\ell,q)=1$. Define
\begin{align}
U(\gamma, \ell, q) 
&:= 
\sum_{n \equiv \ell \modu q} |\mu(n)|\exp(-n\gamma).
\label{eq:def_U_gamma_ell_q}     
\end{align}
Then one has that
\begin{align*}
U(\gamma, \ell, q)
&=
\frac{1}{\zeta(2)\varphi(q)\prod_{p|q}(1+p^{-1})}\frac{1}{\gamma }
+
O\left(\frac{1}{\varphi(q)\gamma_1 (\log (1/\gamma_1)^C)}\right). 
\end{align*}
where $\varphi(q)$ is the Euler totient function and $C\geq 1$ can be chosen arbitrarily.
\end{lemma}
The main step in this proof is to use Abel's summation formula and to employ Lemma~\ref{thm:Siegel_Wal_|mu|}.
Since this calculation is not difficult, we will omit the details.
However, the reader can find for example in \cite[Lemma 6.2]{semiprimes} the complete proof of an analogous result. Next, observe that for all $\beta>0$ we have the elementary but helpful formula
\begin{align*}
e^{-\beta j /X} = \int_{\beta}^\infty jX^{-1}e^{-yj/X}dy.
\end{align*}
Using this identity, we can write 
\begin{align} 
\Phi_{|\mu|}(\rho \e (\alpha)) 
&= 
\sum_{j=1}^\infty \frac{1}{j} \sum_{n=1}^\infty|\mu(n)|\left(\rho \e (\alpha)\right)^{jn}
=
\sum_{j=1}^\infty \frac{1}{j} \int_{1}^\infty j X^{-1}e^{-y j /X } \sum_{n\leq y}|\mu(n)| \e (jn\alpha) \,dy. 
\label{eq:Phionminor}
\end{align}
The last step is to combine \eqref{eq:Phionminor} with Lemma~\ref{lem:main_sum_exp_gamma}.
\begin{lemma}\label{lem:non-principal_arc}
Let $\alpha\in\R$ and $A>0$ be given.
Further, let $a\in\Z$, $q\in\N$ with 
\begin{align*}
(a,q)=1, \quad q\leq (\log X)^A \quad \text{ and } \quad \left|\alpha - \frac{a}{q}\right|\leq q^{-1}X^{-1}(\log X)^A.
\end{align*}
Then there exists $X_0(A)$ such that we have for all $X>X_0(A)$ 
\begin{align}
|\Phi_{|\mu|}(\rho \e(\alpha))|
&\leq
\frac{X}{q}
(1+o(1)).
\label{eq:Phi_on_major_pain}
\end{align}
\end{lemma}
The proof of this lemma is again standard and the interested reader can find, for example, in \cite[Lemma 6.3]{semiprimes} the complete proof of an analogous result.
Lemma~\ref{lem:non-principal_arc} immediately implies
\begin{lemma} \label{lem:nonprincipalarcs}
    Let $\theta \in \mathfrak{M}\backslash \mathfrak{M}(1,0)$. Then we have
    \begin{align}
        \Phi_{|\mu|}(\rho \e (\theta)) \ll_A \frac{3}{4}\Phi_{|\mu|}(\rho) 
    \end{align}
    where $A>0$ is any fixed real number.
\end{lemma}

\subsection{The minor arcs}
In \cite{drzz2} and \cite{gafniPrimePowers}, it is explained that the forthcoming argument is the main step needed to bound the integral over the minor arcs $\mathfrak{m}$. We shall give the broad details. Recall that $\rho=e^{-1/X}$. Truncating the $J$-sum of \eqref{eq:Phionminor} shows that
\begin{align}
    \Phi_{|\mu|}(\rho \e(\theta)) 
    = \sum_{j=1}^J \frac{1}{j}\int_1^\infty jX^{-1}e^{-jt/X} S_{|\mu|}(t,j\theta)dt + O\bigg(\frac{X}{J}\bigg),
\end{align}
where $J = (\log X)^{2}$. For each $j \le J$, Dirichlet's theorem can be used to choose $a \in \Z$ and $q \in \N$ with $(a,q)=1$ such that $|j\theta - \frac{a}{q}| \le q^{-1}X^{-1}(\log X)^A$ and $q < X (\log X)^{-A}$. Next, we set $a_j := a/(a,j)$ and $q_j = jq/(a,j)$. The definition of $\delta_q$ from Section \ref{subsec:arcsDefinition} implies that $|\alpha-\frac{a_j}{q_j}| \le \delta_{q_j}$. Since $\theta \in \mathfrak{m}$, it then follows that $q_j >Q$ where $Q = (\log X)^A$. Once either bound \eqref{eq:new_|mu|} or \eqref{eq:weakerboundmusquared} is established, the rest of the argument follows in a straightforward manner.

\begin{lemma} \label{lem:minorarcs}
    Let $\theta \in \mathfrak{m}$. Then we have
    \begin{align}
        \Phi_{|\mu|}(\rho \e (\theta)) \ll_A X(\log X)^{10-A/4}
    \end{align}
    where $A>0$ is any fixed real number.
\end{lemma}

 For the sake of illustration of the weaker method, the bounds on Lemma \ref{lem:minorarcs} correspond to those from \eqref{eq:weakerboundmusquared}. Lemma \ref{lem:minorarcs} can be strengthened by using Theorem \ref{thm:exponential sum for mu squared} instead but it is immaterial to the rest of the argument.

\subsection{The principal major arcs} \label{sec:sec|mu|major}

Next, we present a common result in partition asymptotics that follows from an application of the saddle point method.

\begin{lemma}
\label{lem:asympt_p_mu}
    Set $\rho=\rho(n)$. One has
    \begin{align}
        \mathfrak{p}_{|\mu|}(n) = \frac{\rho^{-n}\Psi(\rho)}{(2\pi \Phi_{(2)}(\rho))^{1/2}}(1+O(n^{-1/5})).
        \label{eq:asympt_p_mu}
    \end{align}
    where $\Phi_{(m)}(\rho) = (\rho \frac{d}{d\rho})^m \Phi(\rho)$ as given by \textnormal{\eqref{eq:maintermslemmazeros}}
\end{lemma}

In order to get the asymptotics we compute the terms in the numerator and denominator. First
\begin{align}
    \rho^{-n} \Psi(\rho) = \exp\bigg(n \log \frac{1}{\rho(n)} + \Phi_{|\mu|}(\rho(n))\bigg).
\end{align}
We again distinguish two situations. One where no assumption is made on the non-trivial zeros, and then one where their real part is upper bounded by a fixed constant $h<1$. We begin with the unconditional case.
Setting $\theta=m=0$ in Lemma~\ref{lem:lemmaMainTerms} and inserting \eqref{eq:saddle_pint_solution}, we get
\begin{align}
    \Phi_{|\mu|}(\rho(n)) 
    = 
    X +  \log \frac{X}{2\pi}  + \lim_{\nu \to \infty}\sum_{|\imag(\varpi)| < T_\nu} f(X,0,\varpi) + \sum_{n=1}^\infty g(X,0,n)
    = 
    n^{\frac{1}{2}}+O(n^{\frac{1}{4}+\varepsilon}).
\end{align}
Similarly, we obtain
\begin{align}
    n\log \left(1/\rho(n)\right)
    &=
    nX^{-1} 
    =
    n^{\frac{1}{2}}+O(n^{\frac{1}{4}+\varepsilon}) \quad \textnormal{and} \quad
    (\Phi_{(2)}(\rho(n)))^{1/2} 
    =
     n^{\frac{3}{2}}+O(n^{\frac{5}{4}+\varepsilon}).
\end{align}
On the other hand, if $\real(\varpi)\leq h$ for all zeros of $\zeta$ then
\begin{align}
    \Phi_{|\mu|}(\rho(n)) 
    &= 
    X +  O(X^{\frac{h}{2}+\varepsilon})
    =
    n^{\frac{1}{2}} +  O(n^{\frac{h}{4}+\varepsilon}),\\
    n\log \left(1/\rho(n)\right)
    &=
    n^{\frac{1}{2}}+O(n^{\frac{h}{4}+\varepsilon})
    \quad \textnormal{and} \quad
    (\Phi_{(2)}(\rho(n)))^{1/2} 
    =
     n^{\frac{3}{2}}+O(n^{1+\frac{h}{4}+\varepsilon}).
\end{align}
Inserting these expressions into \eqref{eq:asympt_p_mu} completes the structure of Theorem~\ref{thm:partitionTheorem}.
Further, we thus see that the non-trivial zeros of $\zeta$ are visible in \eqref{eq:asympt_p_mu} if one is working with the exact, but implicit saddle point solution. 
On the other hand, if one is trying to determine this solution explicitly and enter it in the above expression then the zeros are absorbed by the error term.

\subsection{Final step}
Combining Lemmas \ref{lem:nonprincipalarcs} and \ref{lem:minorarcs} for the error terms along with Lemmas \ref{lem:lemmaMainTerms} and \ref{lem:asympt_p_mu} for the main terms yields the proof of Theorem \ref{thm:partitionTheorem}.

\section{Conclusion and future work} \label{sec:sec6}

An attractive application of Theorem \ref{thm:mainresult} follows a parallelism with Vinogradov's original goal. We will show in \cite{drzz2}, which is a follow up to this paper, that we can obtain results on partitions into three integers with prescribed arithmetic conditions. For instance let $N$ be an integer and consider the number $r(r_1,r_2,r_3,N)$ defined by
\[
r(r_1,r_2,r_3,N) = \sum_{n_1+n_2+n_3=N} \Lambda^{*r_1}(n_1)\Lambda^{*r_2}(n_2)\Lambda^{*r_3}(n_3).
\]
Retaining the definition and set up of the major and minor arcs from Section \ref{subsec:arcsDefinition}, we shall see that Theorem \ref{thm:mainresult} will allow us to show that 
\begin{align*} 
I_r:=\int_{\mathfrak{m}} \widetilde{S}_r(\alpha, X)^3 \e(-\alpha X) d\alpha \ll \frac{X^{2r-1}}{(\log X)^{A/2-(4+r)}} 
\end{align*}
for any $A>0$ and where the implied constant depends only on $A$. This is an important step in obtaining an asymptotic expression for $r(r_1,r_2,r_3,N)$. Note that if $r_1=r_2=r_3=1$, then this reduces to 
\[
r(1,1,1,N)=r(N) = \sum_{p_1+p_2+p_3=N} \log p_1 \log p_2 \log p_3,
\]
the counting function for the number of representations of an odd integer $N$ as the sum of three primes $p_1, p_2$ and $p_3$. Likewise, one could consider the quantity
\[
    t(N) = \sum_{n_1+n_2+n_3=N} |\mu(n_1)||\mu(n_2)||\mu(n_3)|,
\]
the number of representations of an integer $N$ as the sum of three numbers with prescribed $|\mu|$ conditions. Bounding 
\begin{align*} 
I_{|\mu|}:=\int_{\mathfrak{m}} \widetilde{S}_{|\mu|}(\alpha, X)^3 \e(-\alpha X) d\alpha 
\end{align*}
by the use of Theorem \ref{thm:exponential sum for mu squared} will be a key aspect in determining $t(N)$.

In \cite{drzz2} we will also study on these exponential sums can be used in the Hardy-Littlewood method to bound the minor arcs arising from partitions associated to functions of primes or squarefree numbers. For example, let $P \in \mathbb{Z}[x]$ be a polynomial such that $P(\mathbb{N}) \subseteq \mathbb{N}$. We then define $\mathfrak{p}_{|\mu|,P}(n)$ to be the number of partitions of $n$ whose parts lie in the set $Q_P = \{P(n): n \in \N \textnormal{ and } n \textnormal{ squarefree}\}$. If $\gcd(Q_P)=1$ to avoid undesirable congruence obstructions, then we may extract the value of $\mathfrak{p}_{|\mu|,P}(n)$ by employing Theorem \ref{thm:musquaredpolynomial} on the associated minor arcs, see also \cite{dunnrobles}. Another example takes place when we define the generating function
\[
\Psi_{\Pri_r} (z) := \sum_{n=0}^\infty \mathfrak{p}_{\Pri_r}(n) z^n = \prod_{p_1, \cdots, p_r \in \Pri} \frac{1}{1-z^{p_1 \cdots p_r}} = \exp(\Phi_{\Pri_r}(z)),
\]
to study partitions into $r$-full primes. 
Using Theorem \ref{thm:mainresult} we will also be able to prove that
\[
\Phi_{\Pri_r} (\rho \e(\alpha)) \ll X(\log X)^{3-\frac{A}{2^r}}
\]
for $\alpha \in \mathfrak{m}$ and $A>0$. This will be the key step in bounding the minor arcs which is usually the hardest ingredient of the Hardy-Littlewood method. An adapted Siegel-Walfisz result will handle the non-principal major arcs $\,\mathfrak{M} \backslash \mathfrak{M}(1,0)$. The main terms of $\mathfrak{p}_{\Pri_r}(n)$ will be produced with difficulty employing integration over a Hankel contour to avoid the essential singularities produced by the non-trivial zeros of $\zeta(s)$. This application is lengthy and it is considerably more involved than the treatment in \cite{semiprimes}, so we will not include it in the present paper.  

\section{Acknowlegements}
NR wishes to acknowledge support from Vikram Kilambi and from Christopher Pernin and also wishes to thank Kunjakanan Nath and Debmalya Basak for fruitful discussions. DZ was supported by the Leverhulme Trust Research Project Grant RPG-2021-129.

\bibliographystyle{abbrv}
\bibliography{literature_triple}
\end{document}